\documentclass[11pt, letterpaper]{article}

\usepackage{ahme}
\usepackage{verbatim}

\DeclareMathOperator{\tv}{TV}

\newcommand{\trunc}{\eps}

\newcommand{\ord}[2]{#1_{(#2)}}
\newcommand{\wgt}[2]{\exp \lprp{- C \cdot \frac{#1}{\min (i, n - i)}}}
\DeclareMathOperator{\spread}{spread}

\definecolor{pastelblue}{RGB}{173, 216, 230}
\definecolor{pastelgreen}{RGB}{152, 251, 152}

\definecolor{pastelpurple}{RGB}{216, 191, 216}
\definecolor{pastelorange}{RGB}{255, 204, 153} 
\definecolor{darkpurple}{RGB}{128, 0, 128}
\definecolor{darkorange}{RGB}{255, 140, 0}

\newcommand{\subgestimator}{\ifmmode\text{Sub-Gaussian Mean Estimator}\else Sub-Gaussian Mean Estimator\fi}
\newcommand{\comp}{\ifmmode\text{Comparison}\else Comparison\fi}

\def\com{1}

\if\com1
\newcommand{\Ynote}[1]{\footnote{\color{ForestGreen}Yeshwanth: #1}}
\newcommand{\Dnote}[1]{\footnote{\color{Magenta}Daniel: #1}}
\else
\newcommand{\Ynote}[1]{}
\newcommand{\Dnote}[1]{}
\fi

\newcommand{\yeshwanth}[1]
{
    \if\com1
        \todo[inline,color=green!30]{\small\textbf{Yeshwanth:} #1}
    \else
    \fi
}
\newcommand{\daniel}[1]
{
    \if\com1
        \todo[inline,color=purple!30]{\small\textbf{Daniel:} #1}
    \else
    \fi
}

\begin{document}
\title{Heavy-tailed Contamination is Easier than Adversarial Contamination}
\author{Yeshwanth Cherapanamjeri\thanks{CSAIL, Massachusetts Institute of Technology. \texttt{yesh@mit.edu}} \and Daniel Lee\thanks{CSAIL, Massachusetts Institute of Technology. \texttt{lee\_d@mit.edu}}}

\date{}

\maketitle

\begin{abstract}
    A large body of work in the statistics and computer science communities dating back to Huber (Huber, 1960) has led to the development of statistically and computationally efficient estimators robust to the presence of outliers in data. In the course of these developments, two particular outlier models have received significant attention: the adversarial and heavy-tailed contamination models. While the former models outliers as the result of a potentially malicious adversary inspecting and manipulating the data, the latter instead relaxes the assumptions on the distribution generating the data allowing outliers to naturally occur as part of the data generating process. In the first setting, the goal is to develop estimators robust to the largest fraction of outliers while in the second, one seeks estimators to combat the loss of \emph{statistical} efficiency caused by outliers, where the dependence on the failure probability is paramount.
    
    Surprisingly, despite these distinct motivations, the algorithmic approaches to both these settings have converged, prompting questions on the relationship between the corruption models. In this paper, we investigate and provide a principled explanation for this phenomenon. First, we prove that \emph{any} adversarially robust estimator is also resilient to heavy-tailed outliers for \emph{any} statistical estimation problem with i.i.d data. As a corollary, optimal adversarially robust estimators for mean estimation, linear regression, and covariance estimation are also optimal heavy-tailed estimators. Conversely, for arguably the simplest high-dimensional estimation task of mean estimation, we establish the existence of optimal heavy-tailed estimators whose application to the adversarial setting \emph{requires} any black-box reduction to remove \textit{almost all the outliers} in the data. Taken together, our results imply that heavy-tailed estimation is likely easier than adversarially robust estimation opening the door to novel algorithmic approaches bypassing the computational barriers inherent to the adversarial setting. Additionally, \emph{any} confidence intervals obtained for adversarially robust estimation also hold with high-probability.

    The proof of our reduction from heavy-tailed to adversarially robust estimation rests on the isoperimetry properties of the set of adversarially robust datasets. Meanwhile, to show that the other direction is not possible, we identify novel structural properties on the data sample drawn from a heavy-tailed distribution. We show that such a sample obeys a logarithmic tail-decay condition scaling with the target failure probability. This allows for a quantile-smoothed heavy-tailed estimator which \emph{requires arbitrarily large} stable subsets of the input data to succeed. In the process of analyzing this estimator, we also strengthen the analysis of algorithms utilized previously in the literature. 
\end{abstract}

\thispagestyle{empty}
\setcounter{page}{0}
\newpage

\section{Introduction}
\label{sec:intro}

Over the past several decades, statistical methods have been, and continue to be, employed in an increasingly diverse range of applications. However, classical statistical estimation procedures such as ordinary least squares (OLS) are extremely susceptible to the presence of \emph{outliers} in the data, an increasingly prevalent issue in the large-scale datasets ubiquitous in modern machine learning. Furthermore, the sheer scale of the data used in these systems makes manual data curation extremely challenging. The challenge of mitigating the effects of these outliers has been extensively investigated over the past 60 years by the statistics and computer science communities leading to the development of estimators with near-optimal statistical and algorithmic guarantees.

For several fundamental high-dimensional estimation tasks, such as mean estimation, covariance estimation, and linear regression, much recent work has focused on two distinct models for the presence of outliers in data. In the \textit{adversarial model}, one assumes that a computationally unbounded adversary is allowed to inspect the data and arbitrarily corrupt a fraction of the samples. In contrast, the \textit{heavy-tailed model} relaxes the distributional assumptions on the data, removing the stringent requirements on the higher-order moments of the data distribution\footnote{Typically, these only require the finiteness of a small number of lower-order moments as opposed to much stronger ones such as Gaussianity or sub-Gaussianity which restrict \emph{all} the moments of the distribution.}, allowing the natural occurrence of outliers as part of the data generation process. Surprisingly, despite their distinct motivations, the algorithmic approaches for these settings have converged with several estimators from one applicable to the other. The goal of the current paper is to shed light on this phenomenon by addressing the following fundamental question:
\begin{quote}
    \emph{What is the fundamental relationship between these corruption models that enables algorithmic transference? Are they algorithmically equivalent? Is one corruption model strictly weaker than the other?}
\end{quote}

In this paper, we address the above question by establishing a formal relationship between the two corruption models. For a general class of statistical estimation problems, we show through a black-box analysis that \emph{any} estimator resilient to \emph{adversarial outliers} is also resilient to \emph{heavy-tailed} ones\footnote{In this paper, we restrict our analysis to \emph{deterministic} estimators. This avoids pathological scenarios such as those where the estimator itself may fail with constant probability (say, $0.05$). However, our analysis is also applicable to settings where the failure probability of the \emph{estimator} is chosen to be sufficiently small (say, $\exp (-n)$).}. Conversely, \emph{no} such guarantees hold for the other direction. Even for the arguably simplest robust estimation problem of \emph{high-dimensional mean estimation}, there exist estimators resilient to heavy-tailed outliers that break down when faced with adversarial corruption. In fact, we establish the much stronger statement that \emph{any} black-box reduction, even one which may potentially alter the points presented to the heavy-tailed estimator, must produce a pointset with a negligible fraction of adversarial outliers, effectively filtering out most of the outliers in the dataset. 

\paragraph{Adversarial contamination.}
The adversarial contamination model is a family of models dating back to the work of Huber \cite{huber64}, and the statistical and computational limits of estimation under this model are well understood. The idea is to model outliers via an adversary who is able to corrupt the sample. In particular, in this paper we focus on the \textit{strong contamination model}. In this model, a computationally unbounded adversary is allowed to inspect the entire sample, and arbitrarily change any $\eps$ fraction of the data.

One simple consequence of this definition is that, immediately, most classical estimators (such as empirical mean, empirical covariance, or ordinary least squares) fail to provide even \textit{finite} bounds. For example, even by changing only one point in the sample, the adversary can arbitrarily skew the empirical mean. The difficulty in this model is therefore even getting guarantees with a \textit{constant} probability of success.

\paragraph{Heavy-tailed Model.}
In the heavy-tailed setting, outliers are not modeled through an explicit adversary but rather a natural consequence of relaxing the assumptions on the data distribution. Here, one studies the loss in \textit{statistical efficiency} incurred by relaxing, say a Gaussian or sub-Gaussian assumption on the data, to a broader family of distributions which only restrict a small number of lower-order moments ($2$ or $4$ for the applications in the present paper). By studying the decay of the recovery guarantees with respect to the failure probability, this model allows for outliers to occur as samples from the tail events of the underlying distribution. Surprisingly, in many settings, it is possible to obtain statistical rates which match those obtainable under a sub-Gaussian assumption.

\paragraph{Relationship between the two models.}
Despite the a priori very different motivations and challenges of these two models, there has been a convergence in the algorithmic approaches to solving them. Recent work has constructed algorithms that achieve optimal statistical rates \emph{while} being optimally robust to outliers. In fact, in several cases, algorithms developed for one setting are directly applicable to the other with minimal modifications \cite{deplecmean,lugosi2018risk,mzcovariance,hopkinsRobustHeavyTailedMean2020,mvz,diakonikolasOutlierRobustMean2021,lmtrimmed,ctbj,azcov}. The works of Diakonikolas, Kane, and Pensia \cite{ diakonikolasOutlierRobustMean2021} and Hopkins, Li and Zhang \cite{hopkinsRobustHeavyTailedMean2020} explain this phenomenon by identifying \emph{sufficient} structural properties in samples drawn from these two models that enable this transference of \emph{specific} algorithmic approaches for restricted setting of \emph{mean estimation}.

In contrast, the goal of this paper is to understand the fundamental relationship between these two contamination models that allows algorithms developed for one setting to transfer to the other for a broad range of estimation tasks. We aim to understand the limits of \emph{black-box} reductions (see \cref{sssec:adv_stronger}) between these two models for arbitrary estimation tasks with i.i.d inputs. 

\subsection{Our Contributions}

We now move on to the presentation of our results which are a pair of theorems concerning the black-box reducability of one model to another.

\subsubsection{The adversarial contamination model is at least as strong.}
Our first main result shows that the adversarial contamination model supersedes the heavy-tailed model. We show that \emph{any} adversarially robust algorithm which uses no internal randomness is simultaneously resilient to adversarial contamination while also succeeding with high-probability. In order to state our result in full generality, we formally define the adversarial model we're considering as well as a definition of adversarially robust algorithms for generic estimation problems.

\begin{definition}[Strong Adversarial Contamination Model]\label{def:adversarial_model}
    First, $n$ data points $\wt{\bm{X}} \coloneqq \{X_i\}_{i = 1}^n$ are drawn i.i.d from a distribution $\mc{D}$. An adversary then inspects $\wt{\bm{X}}$ and constructs $\bm{X}_\eps$ by arbitrarily corrupting any $\eps$ fraction of their choosing. We refer to $\bm{X}_\eps$ as an $\eps$-corrupted sample of $\wt{\bm{X}}$.
\end{definition}
\begin{definition}[Generic Adversarially Robust Estimator]
\label{def:adv_rob_est}
    Let $\calD$ be a distribution. $G_\calD(\alpha)\subseteq \calZ$ the set of acceptable solutions, for corruption factor $\eps$. We say that $\calA$ is an adversarially robust estimator, if 
    $$\P_{\bX,\calA}\lbrb{\calA(\bX_\eps;\eps) \in G_\calD(\eps)}\geq 9/10$$
\end{definition}
The definition of $G_\calD (\eps)$ corresponds to the optimal confidence intervals for $\eps$-corruption setting and its definition varies by the particular instantiations of the estimation problem. Note that in the above definition, there's a minimization over all potential adversaries which is left implicit. Our main result here is that asking for constant probability of success \textit{immediately implies} high-probability guarantees (therefore, satisfying the conditions of the heavy-tailed model).

\begin{theorem}[Informal] \label{thm:heavy_from_adversarial_informal}
    For any constant $\eps \in (0, 0.1)$, there exists an absolute constant $c > 0$ such that any adversarially-robust estimation algorithm with no internal randomness satisfying \cref{def:adv_rob_est}, when given corruption parameter $\eps$ and clean sample $\bm{X}$ satisfies:
    \begin{equation*}
        \P \lbrb{\calA(\bX,\eps)\in G_\calD(\eps)} \geq 1- \exp (-c n).
    \end{equation*}
\end{theorem}
As stated previously, we observe that this restriction on having no internal randomness is necessary for a simple reason: one can have an adversarially robust algorithm that just flips a biased coin such that 5 percent of the time it always fails. Clearly, this algorithm will not satisfy exponentially small tail bounds. Unless otherwise stated, we will assume in this paper that algorithms are deterministic and do not rely on any internal sources of randomness.

Note that in the language of black-box reductions, the \emph{trivial} reduction which passes the input to the algorithm directly suffices. Furthermore, the above result places no assumptions on how the adversarially robust algorithm is implemented as opposed to prior work which analyze \emph{specific} algorithms \cite{diakonikolasOutlierRobustMean2021,hopkinsRobustHeavyTailedMean2020} for the restricted setting of mean estimation. 

The proof of \cref{thm:heavy_from_adversarial_informal} rests on the following observation: if an algorithm produces an accurate estimate on $\bm{X}$, it must also produce an accurate estimate on data sets $\bm{X}'$ whose Hamming distance from $\bm{X}$ is at most $0.1n$. With this observation, \cref{thm:heavy_from_adversarial_informal} follows from classical results on isoperimetry and concentration to show that all but an exponentially small fraction (in terms of the product measure over $\mc{D}$) of possible datasets must be within $0.1n$ of the set of good samples.

To illustrate the power of this theorem, we apply it to the problems of mean estimation, covariance estimation, and linear regression, in each case showing that an optimal algorithm in the adversarial contamination model implies an algorithm that satisfies optimal high probability guarantees in the heavy-tailed model. Our results are presented in the extreme regime of constant corruption fraction $\eps$ (corresponding to $\delta \sim e^{-\Theta(n)}$).

\paragraph{Mean estimation.}

Our first application is the canonical estimation problem of high-dimensional mean estimation. Here, we are given access to $n$ i.i.d samples from a distribution with mean $\mu$ and covariance $\Sigma$ with the goal of optimally estimating $\mu$ in Euclidean norm. Perhaps surprisingly, only recently, a long line of work in the computer science literature \cite{diakonikolas2016robust,lai2016agnostic,diakonikolas2017being,deplecmean,hopkinsRobustHeavyTailedMean2020,diakonikolasOutlierRobustMean2021} has culminated in the development of statistically and computationally efficient, sometimes \emph{near-linear} time, estimators satisfying the following optimal guarantee for mean estimation.

\begin{definition}[Optimal Adversarially Robust Mean Estimator]
    \label{def:opt_adv_rob_est_mean}
    We say $\mc{A}$ is an optimal adversarially robust mean estimator if given any $\eps \in [0, 1/10]$, and any $\eps$-corrupted sample of $\wt{\bm{X}}$, drawn i.i.d from $\mc{D}$ produces an estimate $\wh{\mu}$ satisfying:
    \begin{equation*}
        \P_{\bm{X}} \lbrb{\norm{\wh{\mu} - \mu} \leq C \lprp{\sqrt{\frac{\Tr (\Sigma)}{n}} + \sqrt{\norm{\Sigma} \eps}}} \geq \frac{9}{10}
    \end{equation*}
    for some absolute constant $C > 0$.
\end{definition}

An important note here is that, in contrast with the heavy-tailed model to follow, we only ask for a \textit{constant probability guarantee}. The problem isn't getting a high probability bound, it's to get any bound at all.

In contrast, in the heavy-tailed model, some finite bounds are possible on the performance of the empirical mean in this setting. It can be shown that the following bounds are tight for the empirical mean, denoted by $\wh{\mu}_{\mrm{EM}}$:
\begin{equation}
    \E \lsrs{\norm{\wh{\mu}_{\mrm{EM}} - \mu}^2} \leq \frac{\Tr (\Sigma)}{n} \text{ and } \P \lbrb{\norm{\wh{\mu}_{\mrm{EM}} - \mu} \leq \sqrt{\frac{\Tr (\Sigma)}{n\delta}}} \geq 1 - \delta.
\end{equation}
While the first in-expectation result is the optimal achievable, the second is substantially worse than the guarantees obtained in more restricted settings, such as when the data is obtained from a \emph{Gaussian} distribution. In a surprising development, the seminal work of Lugosi and Mendelson \cite{lugosi2017sub} shows that such rates are in fact statistically achievable for the bounded covariance setting as outlined in the following definition:
\begin{definition}
    \label{def:opt_ht_est}
    We say $\mc{A}$ is an optimal heavy-tailed mean estimator if given any $\delta \in [e^{-cn}, 1/2]$, and sample $\bm{X} = \{X_i\}_{i = 1}^n$ drawn i.i.d from $\mc{D}$ returns an estimate, $\wh{\mu}$, satisfying:
    \begin{equation*}
        \P_{\bm{X}, \mc{A}} \lbrb{\norm{\wh{\mu} - \mu} \leq C \lprp{\sqrt{\frac{\Tr (\Sigma)}{n}} + \sqrt{\frac{\norm{\Sigma} \log (1 / \delta)}{n}}}} \geq 1 - \delta
    \end{equation*}
    for some absolute constants $C, c > 0$.
\end{definition}
By comparing the above rates with those of the empirical mean, we see that the dependence on $\delta$ has improved exponentially and furthermore, is decoupled from the dimension term, $\Tr (\Sigma)$. In fact, the rate in \cref{def:opt_ht_est} is known to be unimprovable even if the data were indeed \emph{Gaussian}. That such a rate is achievable in one dimension was independently discovered by several groups \cite{NemYud83,jerrum1986random,alon1999space}. Surprisingly, the extension of these results to high dimensions was only achieved recently by Lugosi and Mendelson \cite{lugosi2017sub}, who designed an inefficient estimator achieving this guarantee. Soon after, Hopkins \cite{hopkins2018sub} devised the first computationally efficient estimator recovering the same guarantees.

We observe that the optimal rates in the \textit{adversarial} model and the \textit{heavy tailed} model coincide under the mapping $\eps \leftrightarrow \frac{\log(1/\delta)}{n}$. In the regime of constant $\eps$ (corresponding to $\delta \sim e^{-\Theta(n)}$), \cref{thm:heavy_from_adversarial_informal} 
implies the surprising corollary that any optimal adversarially robust mean estimation algorithm automatically gives an optimal estimator in the heavy-tailed model.
\begin{corollary}\label{cor:heavy_from_adversarial_informal_mean}
    In the regime of constant $\eps$ and $\delta$ inverse exponential in $n$, any optimal algorithm for mean estimation under adversarial contamination is also optimal in the heavy-tailed contamination model.
\end{corollary}
This follows by instantiating \cref{def:adv_rob_est} and \cref{thm:heavy_from_adversarial_informal} with $\calZ=\mb{R}^d$ and $$G_\calD(\eps)=\lbrb{\hat{\mu}\in \mb{R}^d : \|\hat\mu-\mu\|\leq C\lprp{\sqrt{\frac{\Tr (\Sigma)}{n}} + \sqrt{\norm{\Sigma} \eps}}}$$

For the remaining applications, we'll leave implicit that we consider $\eps$ to be a constant.
\paragraph{Covariance estimation.}
We now consider an application of \cref{thm:heavy_from_adversarial_informal} to covariance estimation, another fundamental statistical estimation task. Here, we are given $n$ i.i.d samples from a high-dimensional distribution with the goal of recovering the covariance matrix of the distribution in spectral norm. We instantiate this task with an $L_4-L_2$ hypercontractivity assumption, standard for high-dimensional heavy-tailed covariance estimation \cite{mzcovariance,azcov}. We formally state this assumption below:
\begin{definition}
    \label{def:l4l2}
    We say that a zero mean random vector $X$ with covariance $\Sigma$ satisfies an $L_4-L_2$ hypercontractivity assumption if:
    \begin{equation*}
        \forall \norm{v} = 1: \lprp{\E [\inp{X}{v}^4]}^{1 / 4} \leq C \sqrt{v^\top \Sigma v}
    \end{equation*}
    for an absolute constant $C > 0$. 
\end{definition}
Note that no further assumptions are made on the distribution of $X$ beyond its $4^{th}$ moments allowing for the existence of heavy-tailed outliers. For this setting, there exist \emph{optimal} estimators \cite{azcov} which when given $n$ i.i.d samples, with $\eps$ fraction arbitrarily corrupted by an adversary, produce estimate, $\wh{\Sigma}$ satisfying with probability at least $1 - \delta$:
\begin{equation*}
    \norm*{\wh{\Sigma} - \Sigma} \leq C\lprp{\sqrt{\frac{\Tr (\Sigma) + \norm{\Sigma} \log (1 / \delta)}{n}} + \sqrt{\norm{\Sigma} \eps}}
\end{equation*}
which is known to be optimal even for the heavy-tailed ($\eps = 0$) and adversarial settings ($\delta$ set to a constant). Applying \cref{thm:heavy_from_adversarial_informal} to the covariance estimation setting, we establish the surprising fact that any optimal algorithm for covariance estimation with adversarial contamination is optimal for heavy-tailed contaminations as well.
\begin{corollary}\label{cor:heavy_from_adversarial_cov}
    Any optimal algorithm for covariance estimation under \cref{def:l4l2} with adversarial contamination is also optimal in the heavy-tailed contamination model.
\end{corollary}
\paragraph{Linear regression.} Our final application is to linear regression. We present a simplified setting for illustrative purposes. However, our results are also applicable more generally. Here, we consider the setting where we are given i.i.d samples generated from a distribution $\calD$ with $(X, Y) \thicksim \calD$ generated as follows:
\begin{equation*}
    Y = \inp{X}{w^*} + \eta
\end{equation*}
where $X$ has zero-mean, covariance $I$, and satisfies \cref{def:l4l2} and $\eta$ is zero-mean, independent of $X$, and has variance bounded by $1$. Our goal now is to recover the unknown parameter $w^*$\footnote{There exist alternative objectives for linear regression such as prediction error. We restrict to the simpler objective of parameter recovery for the sake of exposition. However, our results are also applicable to these alternative objectives.}. Here, again, there exist estimates \cite{lugosi2018risk}, $\wh{w}$, satisfying with probability at least $1 - \delta$:
\begin{equation*}
    \norm{\wh{w} - w^*} \leq C \lprp{\sqrt{\frac{d + \log (1 / \delta)}{n}} + \sqrt{\eps}}
\end{equation*}
which are again known to be optimal for both the adversarial ($\delta$ set to a constant) and heavy-tailed settings ($\eps = 0$). Hence, \cref{thm:heavy_from_adversarial_informal} again yields the following corollary:
\begin{corollary}
    \label{cor:heavy_from_adv_linear_regression}
    Any optimal estimator for linear regression with adversarial contamination is also optimal in the heavy-tailed contamination model.
\end{corollary}

\subsubsection{The adversarial contamination model is strictly stronger.}
\label{sssec:adv_stronger}
\cref{thm:heavy_from_adversarial_informal} prompts the natural question of whether the two models are equivalent? That is, are algorithms for the heavy-tailed setting automatically adversarially robust? We show in a strong sense that the answer to this question is \textbf{no}. The following result shows that, even for the simplest high-dimensional estimation problem of \emph{mean estimation}, \emph{any} reduction from adversarial to heavy-tailed estimation must either remove \emph{most} outliers in the data or must ensure that their \emph{empirical mean} is close to the true mean and furthermore, that their \emph{variance} are small. That is, the reduction itself has to do a significant amount of heavy lifting.

Before we proceed, we define the class of black-box reductions our results apply to below:
\begin{definition}
    \label{def:black_box_reduction}
    For two statistical estimation problems, $P$ and $P'$, we say that $\mc{R}$ is a black box reduction from $P'$ to $P$ if for any estimator $\mathcal{A}$ for $P$, $\mathcal{A} (\mathcal{R} (\cdot))$ is an estimator for $P'$.
\end{definition}
An illustration of this definition is provided in \cref{fig:black_box_reduction}. Here, $\bm{X}$ denotes a sample from the estimation problem $P'$. $\bm{X}$ is then used by the reduction $\mc{R}$ to produce another sample $\bm{Y}$ which is input to the estimator $\mathcal{A}$ for $P$ to produce the final output $\wh{\theta}$. In our context, $P$ and $P'$ typically denote corresponding heavy-tailed and adversarial contamination estimation problems.

We pause for a few remarks on \cref{def:black_box_reduction}. Firstly, note that it captures the setting of \cref{thm:heavy_from_adversarial_informal}. Indeed, \cref{thm:heavy_from_adversarial_informal} establishes that the \emph{trivial} reduction which simply outputs its input, $\bm{X}$, suffices as a black-box reduction from heavy-tailed to adversarially robust estimation. On the other hand, \cref{def:black_box_reduction} places no restrictions on \emph{how} the intermediate output $\bm{Y}$ is produced from $\mc{R}$. Nevertheless, our results establish strong structural properties on $\bm{Y}$ that must be satisfied for \emph{any} black-box reduction from adversarially-robust to heavy-tailed mean estimation. Lastly, observe that the requirements of a black-box reduction are agnostic to the specific estimation algorithms used to complete the reduction. Hence, they allow for a formal investigation into the relationship between different \emph{corruption models}. 

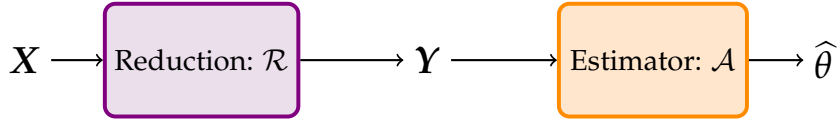
\begin{figure}
\centering
\begin{tikzpicture}[thick, rounded corners]

\node[draw=darkpurple, fill=pastelpurple!50, minimum width=2.5cm, minimum height=1.5cm, rounded corners, line width=1.5pt] (box1) at (0,0) {Reduction: $\mathcal{R}$};

\node[draw=darkorange, fill=pastelorange!50, minimum width=2.5cm, minimum height=1.5cm, rounded corners, line width=1.5pt] (box2) at (6,0) {Estimator: $\mathcal{A}$};

\draw[->] (-2,0) node[left] {\Large $\bm{X}$} -- (box1);

\node (tildeX) at (3,0) {\Large $\bm{Y}$};
\draw[->] (box1) -- (tildeX);

\draw[->] (tildeX) -- (box2);

\draw[->] (box2) -- (8,0) node[right] {\Large $\wh{\theta}$};

\end{tikzpicture}
\caption{Illustration of the class of reductions \cref{thm:sub_g_informal} applies to. The input dataset $\bm{X}$ is processed by the reduction, $\mc{R}$ to produce another dataset $\bm{Y}$ that is ultimately used as input to the estimator $\mc{A}$ which produces the final output $\wh{\theta}$.}
\label{fig:black_box_reduction}
\end{figure}

The exposition of our result additionally requires the concept of \emph{stability}, a core component of recent adversarially robust estimation algorithms \cite{diakonikolas2019recent,dkp20,hlz20}:
\begin{definition}
    \label{def:stability}
    A finite set $S = \{y_i\}_{i = 1}^n \subset \R^d$, is $(\gamma, \nu)$-stable with respect to $\mu \in \R^d$ for $\gamma > 0, \nu \in (0, 1/10)$ if there exists $S' \subset S$ with $\abs{S'} \geq (1 - \nu) n$ such that:
    \begin{equation*}
        \norm*{\frac{1}{\abs{S'}} \sum_{i \in S'} y_i - \mu} \leq \gamma \text{ and } \norm*{\frac{1}{\abs{S'}} \sum_{i \in S'} (y_i - \mu)(y_i - \mu)^\top} \leq \gamma^2. 
    \end{equation*}
\end{definition}
The existence of such stable sets plays a key role in designing adversarially robust estimators. They may be shown to exist with high-probability for a chosen subset of the \emph{inlier} data points. Hence, in the adversarial contamination setting, $\nu > \eps$. In \cite{dkp20}, it is shown that any \emph{stability-based} mean-estimation algorithm is simultaneously robust to adversarial and heavy-tailed corruptions. In fact, virtually all known adversarially robust estimators rely on such stability properties of the inlier data points and typically operate by iteratively pruning the dataset until one is left with a large stable subset of the data whose mean is a good estimate of the population mean. Conversely, the presence of (almost optimal) heavy-tailed estimators that do \emph{not} rely on stability \cite{catoni2018dimension} is suggestive that heavy-tailed outliers are qualitatively weaker than adversarial outliers. 

The main technical result of this section provides a formal justification for the separation between adversarial and heavy-tailed corruptions through the lens of black-box reductions. Concretely, we show that \emph{any} black-box reduction from adversarial to heavy-tailed estimation must produce pointsets with arbitrarily large stable sets.
\begin{theorem}[Informal -- see \cref{thm:gen_conc}]\label{thm:sub_g_informal}
    \emph{Any} black-box reduction (\cref{def:black_box_reduction}) from adversarially robust mean estimation (with $\eps = 0.1$) to heavy-tailed mean estimation (with $\log (1 / \delta) = \Omega (n)$), must produce a pointset $\bm{Z} = \{z_i\}_{i = 1}^m$ with $\mc{I} \subset [m]$ satisfying $\abs{\mc{I}} \geq 0.99m$ and:
    \begin{equation*}
        \norm*{\frac{1}{\abs{\mc{I}}} \sum_{i \in \mc{I}} z_i - \mu} \leq C \lprp{\sqrt{\frac{\Tr (\Sigma)}{n}} + \sqrt{\norm{\Sigma}}} \text{ and } \norm*{\frac{1}{\abs{\mc{I}}} \sum_{i \in \mc{I}} (z_i - \mu) (z_i - \mu)^\top} \leq C \lprp{\frac{\Tr (\Sigma)}{n} + \norm{\Sigma}}.
    \end{equation*}
    That is, this stable set must account for all but a small-constant fraction of the outliers.
\end{theorem}
We briefly remark on \cref{thm:sub_g_informal} to simplify the interpretation of its implications. As previously stated, adversarially corrupted pointsets are only guaranteed to contain stable sets (with high probability) of size at most $0.9n$\footnote{For $\eps = 0.1$, it is easy to construct adversarial contaminations such that all subsets larger than $0.9n$ are \emph{not} stable.}. Hence, the pointsets that are produced by \emph{any} black-box adversarial-to-heavy-tailed reduction are qualitatively different from adversarially corrupted datasets. Furthermore, several algorithmic approaches for adversarially robust estimation rely on the extraction of a large stable subset of the input point set. Hence, the above result establishes that any such reduction must already produce a pointset, essentially all of which is a stable subset. Consequently, in the setting where the reduction passes \emph{subsets} of the input to the heavy-tailed estimation algorithm, it must either \emph{filter} out adversarial points or alternatively, ensure that all but a \emph{negligible fraction} of them are part of a stable set. Hence, the reduction itself operates similarly to existing stability-based algorithms for adversarially robust estimation.

To establish \cref{thm:sub_g_informal}, we construct the novel heavy-tailed estimator that provides the guarantees of \cref{thm:heavy_from_adversarial_informal}. The design of our estimator leverages a novel tail-decay condition that is guaranteed to hold when the data is drawn from a distribution but does \emph{not} hold when an adversary engineers a large fraction of the outlier set to deviate from the mean. The estimator then estimates a quantile \emph{smoothed} notion of spread over all one-dimensional projections of the data and adds a random perturbation of this magnitude to the output of a standard heavy-tailed estimator. However, care must be taken in the choice of the underlying estimator. To avoid cancellations in magnitude to establish \cref{thm:sub_g_informal} on the points provided to the algorithm as \emph{input}, we use a standard stability-based framework for robust estimation \cite{steinhardt2017resilience}. Unfortunately, these estimators are not known to be optimal heavy-tailed estimators with prior analyses incurring an additional multiplicative logarithmic factor \cite{dkp20}. We provide a novel analysis of this estimator by providing a fine-grained analysis of a Gaussian rounding scheme previously used to analyze a median-of-means approach \cite{deplecmean} and leveraging recent developments in the analysis of the trimmed median \cite{lmtrimmed}.

\paragraph*{Notation.} In this paper, $\R$ and $\N$ denote the set of real and natural numbers respectively. For $n \in \N$, $[n]$ denotes the set of integers from $1$ through $n$. For $d, m, n \in \N$, $\R^d$ and $\R^{m \times n}$ denote the set of $d$-dimensional vectors and $m \times n$ dimensional matrices respectively. For $v \in \R^d$ and $M \in \R^{d \times d}$, $\norm{v}$ and $\norm{M}$ are the standard Euclidean and spectral norms respectively while $\Tr (M)$ denotes the trace of $M$. For measure $\mc{D}$ and $n \in \N$, $\mc{D}^{\otimes n}$ denotes the $n$-time produce measure of $\mc{D}$. For a set $T \subset \R^d$, $\mrm{mean} (T)$ denotes the mean of the uniform distribution over the set and when $d = 1$, $\mrm{Med} (T)$ denotes its median. For $x \in \R$, $\sgn(x)$ denotes the signum function of $x$. For distributions $\mc{D}_1, \mc{D}_2$, $\mrm{TV} (\mc{D}_1, \mc{D}_2)$ denotes the TV distance between them. For $v \in \R^d$, $v_i$ denotes the $i^{th}$ entry of $v$. For $\mc{S} \subset \R^d$, $\mrm{Unif} (\mc{S})$ denotes the uniform distribution over $\mc{S}$. Finally, for an event $E$, $\bm{1} \lbrb{E}$ denotes the indicator of that event.

\paragraph{Organization.} In the remainder of the paper, we establish \cref{thm:heavy_from_adversarial}, the formalization of \cref{thm:heavy_from_adversarial_informal}. Subsequently, in \cref{sec:sub_g}, we design and analyze the sub-Gaussian estimator used in the proof of \cref{thm:sub_g_informal} whose generalized counterpart is \cref{thm:gen_conc}.

\section{From heavy-tailed to adversarially robust X estimation}\label{sec:heavy_from_adversarial}

Here, we prove that an algorithm, robust to a constant factor of adversarial corruptions, with no additional assumptions required and no additional work, automatically achieves sub-Gaussian guarantees. The argument reduces to a statement about isoperimetry in the hamming metric. Recall that in the adversarial corruption setting, we require that our algorithm be robust to arbitrary $\eps n$ sized corruptions of the sample with constant probability. That is, there is some set with constant mass such that, for any point within $\eps n$ hamming distance of this set, the algorithm outputs a ``good" mean estimate. The question becomes, what is the measure of this $\eps n$ hamming ball when the data is drawn from a heavy-tailed distribution? Classical results on the connection between isoperimetry and concentration show that this is all but an exponentially small fraction.

For an abstract space, $\mc{X}$, consider the set of possible samples obtained from $n$ iid draws from a distribution over $\mc{X}$, $\mc{X}^n$. For instance, in our setting $\mc{X} = \R^d$. For two possible sample sets, $x, y \in \mc{X}^n$ and $S \subset \mc{X}^n$, define the Hamming distance as follows:
\begin{equation*}
    \dist_{\mrm{ham}}(x,y):=\sum_i \mathbb{1}[x_i-y_i\neq 0] \text{ and } \dist_{\mrm{ham}}(x,S)=\min_{y\in S}\dist_{\mrm{ham}}(x,y).
\end{equation*}

We recall the following fact about expansion with respect to $\dist_{\mrm{ham}}$.
\begin{lemma}[\cite{boucheron2013concentration} Corollary 7.4]\label{lem:boucheron}
    Let $S\subseteq \calX^n$ and $\nu$ a measure on $\calX$ s.t. $\nu^{\otimes n}(S) = p$. We have: 
    \begin{equation*}
        \forall t \geq 0: \P\lbrb{\dist_{\mrm{ham}}(X,S)\geq t + \sqrt{\frac{n}{2}\log(1/p)}}\leq e^{-2t^2/n}.
    \end{equation*}
\end{lemma}

As a simple corollary, we obtain the following lower bound on the blowup of a set.
\begin{corollary}\label{lem:blowup}
    Let $S \subseteq \mc{X}^n$ and $\nu$ be a distribution over $\mc{X}$ be such that:
    \begin{equation*}
        \nu^{\otimes n}(S)\geq 0.9.
    \end{equation*} 
    Then,
    \begin{equation*}
        \forall \eps \geq 0: \nu^{\otimes n}(S_\eps)\geq 1-2e^{-\eps^2n} \text{ where } S_\eps := \{x\in\calX^n: \dist_{\mrm{ham}}(x,S)\leq \eps n\}.
    \end{equation*}
\end{corollary}
\begin{proof}
    Set $t = \eps n - \sqrt{\frac{n}{2}\log(1/p)}$ in \cref{lem:boucheron}. 
    \begin{align*}
        \nu^{\otimes n}(S_\eps) = \P\lbrb{\dist_{\mrm{ham}}(X,S)\geq \eps n}
        &\leq \exp\lbrb{-\frac{2\lprp{\eps n - \sqrt{\frac{n}{2}\log(1/p)}}^2}{n}}
    \end{align*}

    \noindent Let $\alpha > 0$ be a parameter to be determined subsequently. We now have two cases:

    \paragraph{Case 1:} $n\geq \frac{1}{\eps^2 2(1-\alpha)^2}\log(1/p)$. This implies that $\sqrt{\frac{n}{2}\log(1/p)}\leq \eps n(1-\alpha)$. Plugging in, we get $$\P\lbrb{\dist_{\mrm{ham}}(X,S)\geq \eps n}\leq e^{-2\eps^2 \alpha^2 n}$$

    \paragraph{Case 2:} $n< \frac{1}{\eps^2 2(1-\alpha)^2}\log(1/p)$. Then, by setting $\alpha = 1 / \sqrt{2}$
    \begin{align*}
        2e^{-2 \eps^2 \alpha^2 n}
        &\geq 2\exp\lbrb{-\log(1/p)\frac{\alpha^2}{(1-\alpha)^2}} = 2p^{\frac{\alpha^2}{(1-\alpha)^2}} \geq 1 \geq 1-\nu^{\otimes n}(S_\eps),
    \end{align*}
    concluding the claim.
\end{proof}

\cref{lem:blowup} is precisely what we need to formalize our argument that the robustness of our algorithm ensures all but a vanishingly small fraction of the samples will be good.

\begin{theorem}\label{thm:heavy_from_adversarial}
     Let $\calA$ be an adversarially robust algorithm with no internal randomness satisfying \cref{def:adv_rob_est}. Then, given samples $\bX=\{x_i\}_{i=1}^n$ from a distribution, $\mc{D}$ and corruption parameter $\eps\in(0,1/10)$, there exists a constant $c(\eps)>0$ s.t.
     \begin{equation*}
        \hat{\theta}\in G_\calD(\eps)
     \end{equation*}
    with probability $1-2e^{-c(\eps)n}$, where $\hat{\theta}=\calA(\bX;\eps)$. More generally, for $\eps' \in [0,\eps)$, $\calA$ outputs $\hat{\theta}$ satisfying the above guarantee with probability $1-2e^{-n(\eps-\eps')^2}$ when the sample it receives is $\eps'$-corrupted.
\end{theorem}
\begin{proof}
   Let $A(\eta)\subseteq (\R^d)^n$ denote the set of samples on which $\calA$ satisfies $\eta$ robustness. Let $S=\calA(\eps)$. We first observe that for any $\eta,\eta',\alpha > 0$ where $\alpha = \eta - \eta'$: 
    \begin{equation*}
        A(\eta)_\alpha \subseteq A(\eta').
    \end{equation*}
   That is, the $\alpha$ blowup of the set on which $\calA$ satisfies $\eta$ robustness is contained in the set on which $\calA$ satisfies $\eta'$ robustness. This follows from the triangle inequality for $\dist_{\mrm{ham}}$.
   
   Now, by assumption, $\mc{D}^{\otimes n}(S)\geq 0.9$. Let $S_\alpha$ denote its $\alpha$ blowup, instantiating $\alpha=\eps - \eps'$. By \cref{lem:blowup}, we know that $\mc{D}^{\otimes n}(S_\alpha)\geq 1-2e^{-\alpha^2 n}$. From our above observation, $S_\alpha$ is a subset of the points on which $\calA$ is $\eps'$ robust; i.e. the samples on which, even with an adversary who can corrupt $\eps'$ of the samples, $\calA$ outputs $\hat{\theta}$ satisfying: 
   \begin{equation*}
    \theta\in G_\calD(\eps)
   \end{equation*} 
    This observation along with the fact that $\mc{D}^{\otimes n} (S_\alpha) \geq 1 - 2e^{-\alpha^2 n}$ conclude the proof.
\end{proof}

\section{A Hard sub-Gaussian Algorithm}\label{sec:sub_g}

In this section, we overview and present guarantees on a robust estimation algorithm that is guaranteed to achieve sub-Gaussian rates while, at the same time, not robust to \emph{adversarial} perturbation. The formal details of the algorithm are presented in \cref{alg:sub_g_main}. \cref{alg:sub_g_main} parameterizes two algorithms, one for each choice of $s$, both of which are guaranteed to have sub-Gaussian performance. The algorithm exploits a logarithmic tail-decay condition whose strength scales with the targeted failure probability. This condition is used to extract a notion of scale that remains bounded when the data is obtained from a heavy-tailed distribution but does not hold when the data is corrupted by an adversary.

The algorithm operates by first finding a point, $x$, such that there exists a large subset of points (a $1 - \trunc$ fraction) of points where $\trunc \approx \log (1 / \delta) / n$ having bounded second-moment with respect to $x$. That is it solves the following program:
\begin{gather*}
    \argmin_{x \in \R^d} \min_{w \in \mc{W}_\eps} \norm*{\sum_{i = 1}^n w_i (x_i - x)(x_i - x)^\top} \\
    \text{where } \forall \rho \geq 0: \mc{W}_\rho \coloneqq \lbrb{\lbrb{w_i}_{i = 1}^n: 0 \leq w_i \leq \frac{1}{(1 - \rho) n} \text{ and } \sum_{i = 1}^n w_i = 1}.
\end{gather*}
For the next step, define the following notion of variation, denoted by $\spread$, for a one-dimensional point set, $\bm{Y} = \{y_i\}_{i = 1}^n$ with respect to a centering $y$ and $\ord{y}{i}$ denoting the ordered elements of $\bm{Y}$:
\begin{gather*}
    \forall i \in [n]: \wt{w}_i = \wgt{n\trunc}{i} \text{ and } w_i \coloneqq \frac{\wt{w}_i}{\sum_{i = 1}^n \wt{w}_i} \\
    \spread (\bm{Y}, y) \coloneqq \sqrt{\sum_{i \in [n]} w_i (\ord{y}{i} - y)^2}.
\end{gather*}
The algorithm then searches over directions for one with large spread. The final output of the algorithm is a random perturbation along the direction found in the previous step whose scale depends on the value of the variation along that direction and the failure probability. As we will see, this notion of spread forces the algorithm to be successful only when there exists a stable set of size $(1 - c \trunc)$ for any $c < 1$. The tail decay property of samples drawn from heavy-tailed distributions bounds the performance of the estimator in this setting but this may not necessarily hold for the adversarial contamination setting.

\begin{algorithm}[H]
    \caption{\comp}
    \label{alg:comp}
    \begin{algorithmic}[1]
        \State \textbf{Input}: Vectors $v$
        \If {$v = \bm{0}$}
            \State \textbf{Return: } $v$
        \Else
            \State $i^* = \min \{i: v_i \neq 0\}$
            \State \textbf{Return: } $\sgn (v_{i^*}) v$
        \EndIf
    \end{algorithmic}
\end{algorithm}

\begin{algorithm}[H]
    \caption{\subgestimator}
    \label{alg:sub_g_main}
    \begin{algorithmic}[1]
        \State \textbf{Input}: Point set $\bm{X} = \{x_i\}_{i = 1}^n \subset \R^d$, Failure Probability $\delta$, {\color{red} Version $s \in \{+1, -1\}$}
        \State $\trunc \gets \log (4 / \delta) / n$
        \State $\wt{\mu} \gets \argmin_{x \in \R^d} \min_{w \in \mc{W}_{\trunc}} \norm*{\sum_{i = 1}^n w_i (x_i - x)(x_i - x)^\top}$
        \State $v \gets \argmax_{\norm{u} = 1} \spread \lprp{\{\inp{x_i}{u}\}_{i = 1}^k, \inp{\wt{\mu}}{u}}, \sigma_v \gets \spread \lprp{\{\inp{x_i}{v}\}_{i = 1}^n, \inp{\wt{\mu}}{v}}$
        \State $v \gets \comp (v)$
        \State $\wh{\mu} = \wt{\mu} + {\color{red} s} \sqrt{\eps} \sigma_v v$
        \State \textbf{Return: } $\wh{\mu}$
    \end{algorithmic}
\end{algorithm}

\subsection{Proof that \cref{alg:sub_g_main} achieves sub-Gaussian Rates}
\label{ssec:sub_g_rates}

In this section, we will show that \cref{alg:sub_g_main} achieves sub-Gaussian rates on heavy-tailed distribution. The main technical difficulty in this section is that the notion of spread used in \cref{alg:sub_g_main} remains suitably bounded. For this step, we prove a tail-decay condition on samples drawn from heavy-tailed distributions and show that this suffices to establish the required bounds. In addition, even \emph{without} the additional spread detection step, it is not known whether the stability-based intermediate estimate ($\wt{\mu}$ in \cref{alg:sub_g_main}) achieves sub-Gaussian rates. Unfortunately, prior analyses of the estimator are sub-optimal by a logarithmic factor \cite{hopkinsRobustHeavyTailedMean2020,diakonikolasOutlierRobustMean2021}. We provide an improved statistical analysis of these estimators by proving tighter bounds on the truncated variances of the one-dimensional point sets obtained by projecting the dataset onto one-dimensional subspaces, and secondly, by improving the analysis of a Gaussian rounding scheme of Depersin and Lecue \cite{deplecmean}. Our analysis of the performance of the intermediate estimate $\wt{\mu}$ borrows heavily from the analysis of Lugosi and Mendelson \cite{lmtrimmed} and may be viewed as a generalization of their results to the setting of stability-based estimators, where the set of points to be truncated are chosen uniformly for every direction. The main theorem of the section is the following which shows that \cref{alg:sub_g_main} achieves sub-Gaussian rates.

\begin{theorem}
    \label{thm:sub_g_main}
    There exists an absolute constants $C, c > 0$ such that the following hold. Let $X_1, \dots, X_n$ be drawn iid from a distribution with mean $\mu$ and covariance $\Sigma$. Then, for any $\delta \in [e^{-cn}, 1/4]$, the outputs of \cref{alg:sub_g_main}, $\mu_{+}$ and $\mu_{-}$, on input $X_1, \dots, X_n$, $\delta$, and $s = \{+1, -1\}$ respectively satisfy:
    \begin{equation*}
        \max \lbrb{\norm{\wh{\mu}_{+} - \mu}, \norm{\wh{\mu}_{-} - \mu}} \leq C \lprp{\sqrt{\frac{\Tr (\Sigma)}{n}} + \sqrt{\frac{\norm{\Sigma} \log (1 / \delta)}{n}}}
    \end{equation*}
    with probability at least $1 - \delta$.
\end{theorem}

Throughout this section, we assume that $X_1, \dots, X_n$ are drawn iid from a distribution with mean $0$ and covariance $\Sigma$. Furthermore, noting that \cref{alg:sub_g_main} is shift and scale-equivariant, it suffices to consider the case where $\mu = 0$ and $\norm{\Sigma} = 1$. We start by defining a truncation parameter and function below following the analysis of Lugosi and Mendelson \cite{lmtrimmed}:
\begin{equation*}
    \tau \coloneqq C_1 \max \lbrb{\frac{1}{\trunc} \sqrt{\frac{\Tr (\Sigma)}{n}}, \sqrt{\frac{\norm{\Sigma}}{\trunc}}} \text{ and } \forall \gamma \geq 0: \phi_\gamma (x) = 
    \begin{cases}
        x &\text{if } \abs{x} \leq \gamma \\
        \sgn (x) \gamma &\text{otherwise}
    \end{cases}.
\end{equation*}

The next lemma shows that the second moment of the truncated one-dimensional pointsets obtained by projecting the dataset onto any dimension is small. This lemma is used to analyze both the performance of the stability-based intermediate estimate, $\wt{\mu}$, and the scale of the variation, $\sigma_v$, utilized in the algorithm. Since the proof is a simple adaptation of the ideas of Lugosi and Mendelson \cite{lmtrimmed}, its proof is deferred to \cref{sec:deferred_proofs}. Intuitively, the lemma is used to bound the variance of the points whose projections are \emph{small} while the logarithmic tail bound established subsequently bounds the variance of the points with \emph{large} projections.  

\begin{restatable}{lemma}{specnrmbnd}
    \label{lem:spec_norm_bnd}
    We have:
    \begin{equation*}
        \forall \norm{v} = 1: \sum_{i = 1}^n \phi_{\tau} (\inp{X_i}{v})^2 \leq C \cdot \lprp{\frac{\Tr (\Sigma)}{\trunc} + n \norm{\Sigma}}.
    \end{equation*}
    with probability at least $1 - \exp (-\trunc n)$.
\end{restatable}

The next lemma establishes the logarithmic tail bound alluded to previously. We show that the tail of the empirical distribution over samples scales as $\trunc / \log (t)$ with probability at least $1 - e^{-n\trunc}$. This is then exploited to bound the scale of the spread estimate, $\sigma_v$, in \cref{alg:sub_g_main}.
\begin{lemma}
    \label{lem:quant_conc}
    We have:
    \begin{equation*}
        \forall \norm{v} = 1, j \geq 1: \sum_{i = 1}^n \bm{1} \lbrb{\abs{\inp{X_i}{v}} \geq \tau_j} \leq \frac{2 (n\trunc + j)}{3j} \text{ where } \tau_j = \frac{C_1}{64} \cdot e^{4j} \cdot \max \lbrb{\frac{1}{\trunc} \sqrt{\frac{\Tr (\Sigma)}{n}}, \frac{1}{\sqrt{\trunc}}} 
    \end{equation*}
    with probability at least $1 - e^{-\trunc n}$.
\end{lemma}
\begin{proof}
    First define:
    \begin{equation*}
        \psi_j (x) \coloneqq 
        \begin{cases}
            1 &\text{if } \abs{x} \geq \tau_j \\
            0 &\text{if } \abs{x} \leq \frac{\tau_j}{2} \\
            \frac{2\abs{x}}{\tau_j} - 1 &\text{otherwise}
        \end{cases}.
    \end{equation*}
    As in the proof of \cref{lem:spec_norm_bnd}, define random variable:
    \begin{equation*}
        Z \coloneqq \max_{v = 1} \sum_{i = 1}^n W_{i, v} \text{ where } W_{i, v} \coloneqq \psi_j (\inp{X_i}{v}) - \E_{X \thicksim \mc{D}} [\psi_j (\inp{X}{v})].
    \end{equation*}
    We get where $X'_i$ and $\gamma_i$ are independent copies of $X_i$ and independent Rademacher random variables respectively:
    \begin{align*}
        \E [Z] &= \E \lsrs{\max_{\norm{v} = 1} \sum_{i = 1}^n \psi_j (\inp{X_i}{v}) - \E_{X \thicksim \mc{D}} [\psi_j (\inp{X}{v})]} \\
        &\leq \E_{X_i, X_i'} \lsrs{\max_{\norm{v} = 1} \sum_{i = 1}^n \psi_j (\inp{X_i}{v}) - \psi_j (\inp{X_i'}{v})} \\
        &\leq \E_{X_i, X_i', \gamma_i} \lsrs{\max_{\norm{v} = 1} \sum_{i = 1}^n \gamma_i (\psi_j (\inp{X_i}{v}) - \psi_j (\inp{X_i'}{v}))} \\
        &\leq 2\E_{X_i, \gamma_i} \lsrs{\max_{\norm{v} = 1} \sum_{i = 1}^n \gamma_i \psi_j (\inp{X_i}{v})} \leq \frac{8}{\tau_j} \E_{X_i, \gamma_i} \lsrs{\max_{\norm{v} = 1} \sum_{i = 1}^n \gamma_i \inp{X_i}{v}} \\
        &\leq \frac{8}{\tau_j} \E_{X_i, \gamma_i} \lsrs{\norm*{\sum_{i = 1}^n \gamma_i X_i}} \leq \frac{8}{\tau_j} \sqrt{\E_{X_i, \gamma_i} \lsrs{\norm*{\sum_{i = 1}^n \gamma_i X_i}^2}} = \frac{8}{\tau_j} \sqrt{n \Tr (\Sigma)},
    \end{align*}
     the fourth inequality following from the Ledoux-Talagrand contraction inequality (\cref{cor:ledtal}) and noting that $\psi_j$ is $2 / \tau_j$-Lipschitz. Furthermore, observe that from Chebyshev's inequality:
    \begin{equation*}
        \forall \norm{v} = 1 : \E [\psi_j (\inp{X_i}{v})^2] \leq \E \lsrs{\bm{1} \lbrb{\abs{\inp{X_i}{v}} \geq \frac{\tau_j}{2}}} \leq \frac{4}{\tau_j^2}.
    \end{equation*}
    Hence, we get:
    \begin{equation*}
        \E [Z] \leq \frac{8}{\tau_j} \sqrt{n \Tr (\Sigma)} \text{ and } \forall \norm{v} = 1 : \E \lsrs{\psi_j (\inp{X_i}{v})^2} \leq \frac{4}{\tau_j^2}.
    \end{equation*}
    Defining
    \begin{equation*}
        \nu \coloneqq 32 \max \lbrb{\frac{\sqrt{n \Tr \Sigma}}{\tau_j}, \frac{n}{\tau_j^2}} \text{ and } r_j \coloneqq \frac{n \trunc + j}{2j},
    \end{equation*}
    we get from \cref{thm:bousequet_thm}
    \begin{equation*}
        \P \lbrb{Z \geq r_j} \leq \exp \lbrb{- \nu h \lprp{\frac{r_j}{\nu}}} = \exp \lbrb{- \lprp{(\nu + r_j) \log \lprp{1 + \frac{r_j}{\nu}} - r_j}}.
    \end{equation*}
    Now, we get:
    \begin{align*}
        (\nu + r_j) \log \lprp{1 + \frac{r_j}{\nu}} - r_j &\geq r_j \log \lprp{\frac{r_j}{\nu}} - r_j \geq r_j \log \lprp{\frac{C e^{4j}}{32 n\trunc} \cdot r_j} - r_j \geq r_j \log \lprp{\frac{C e^{4j - 1}}{32 n\trunc} \cdot r_j} \\
        &= r_j (2j) + r_j \log \lprp{\frac{C e^{2j - 1}}{32 n\trunc} \cdot r_j} \geq n\trunc + j + r_j \log \lprp{\frac{C e^{2j - 1}}{64j}} \geq n\trunc + j.
    \end{align*}
    Therefore, we get:
    \begin{equation*}
        \forall j \geq 1: \P \lbrb{Z_j \geq r_j} \leq \exp \lbrb{- n\trunc + j}.
    \end{equation*}
    A union bound now yields:
    \begin{equation*}
        \P \lbrb{\exists j \geq 1: Z_j \geq r_j} \leq e^{-n\trunc}.
    \end{equation*}
    Observing that:
    \begin{equation*}
        \forall \norm{v} = 1: \E_{X \thicksim \mc{D}} \lsrs{\psi_j (\inp{X}{v})} \leq \frac{\eps}{C e^{8j}}
    \end{equation*}
    the conclusion of the lemma follows.
\end{proof}

The next lemma shows that the value of the spread estimate $\sigma_v$ in \cref{alg:sub_g_main} is small. It utilizes the results of \cref{lem:spec_norm_bnd} to bound the contribution of points with \emph{small} projection along $v$ while \cref{lem:quant_conc} bounds the contribution of the large points as the guarantees of \cref{lem:quant_conc} are sub-optimal when applied to the smaller projections.

\begin{lemma}
    \label{lem:wt_sp_bnd}
    We have:
    \begin{equation*}
        \forall \norm{v} = 1 : \spread \lprp{\lbrb{\inp{X_i}{v}}_{i = 1}^n, 0} \leq C \lprp{\sqrt{\frac{\Tr (\Sigma)}{\trunc n}} + \sqrt{\norm{\Sigma}}}
    \end{equation*}
    with probability at least $1 - 2e^{-n\trunc}$.
\end{lemma}
\begin{proof}
    As before, it suffices to restrict to the setting $\norm{\Sigma} = 1$. First, condition on the events in the conclusions of \cref{lem:spec_norm_bnd,lem:quant_conc}. Now, fix a particular $v$ with $\norm{v} = 1$ and define $\bm{Y}$ to be the \emph{ordered} set of elements $\{\inp{X_i}{v}\}_{i = 1}^n$. Then, we have:
    \begin{equation*}
        \forall \frac{n}{4} \leq i \leq \frac{3n}{4}: \wt{w}_i \geq c \implies \sum_{i = 1}^n \wt{w}_i \geq cn \implies \forall j \in [n] : w_j \leq \frac{1}{cn}.
    \end{equation*}
    We will now bound the spread separately for the subsets $\mc{G} = \bm{Y} \cap [-\tau, \tau]$ and $\mc{B} = \bm{Y} \setminus \mc{G}$. For the first set, we have by \cref{lem:spec_norm_bnd}:
    \begin{equation*}
        \sum_{i \in \mc{G}} w_i y_i^2 \leq \frac{1}{cn} \sum_{i \in \mc{G}} y_i^2 \leq C \cdot \lprp{\frac{\Tr (\Sigma)}{\eps} + 1}.
    \end{equation*}
    For the second, we further decompose and bound the sum as follows with $\mc{B}_j = \bm{Y} \cap (-\tau_{j + 1}, -\tau_j] \cup [\tau_j, \tau_{j + 1})$ and noting that $\mc{B} \cup_{j = 1}^{4n\trunc} \mc{B}_j$ as $\mc{B}_j = \phi$ for $j < 4n\trunc$ by \cref{lem:quant_conc}:
    \begin{align*}
        \sum_{y \in \mc{B}} w_y y_i^2 &\leq \sum_{j = 1}^{4n\trunc} \sum_{y \in \mc{B}_j} w_y y^2 \leq \frac{1}{cn} \sum_{j = 1}^{4n\trunc} \sum_{y \in \mc{B}_j} \wt{w}_y \tau_{j + 1}^2 = \frac{C}{n} \sum_{j = 1}^{4n\trunc} \sum_{y \in \mc{B}_j} \exp\lprp{-C \cdot \frac{3nj\trunc}{2(n\trunc + j)}} \tau_{j + 1}^2 \\
        &\leq \frac{C}{n} \sum_{j = 1}^{4n\trunc} \frac{2 (n\trunc + j)}{3j} \cdot \exp\lprp{-C j} \tau_j^2 \leq \frac{C}{n} \sum_{j = 1}^{4n\trunc} \frac{n\trunc + j}{j} \cdot \exp(-C j) \cdot e^{8j} \lprp{\frac{1}{\trunc^2} \cdot \frac{\Tr (\Sigma)}{n} + \frac{1}{\trunc}} \\
        &\leq \frac{C}{n} \sum_{j = 1}^{4n\trunc} (n\trunc + j) \cdot \exp(-j) \cdot \lprp{\frac{1}{\trunc^2} \cdot \frac{\Tr (\Sigma)}{n} + \frac{1}{\trunc}} \leq \frac{C}{n} \lprp{\sum_{j = 1}^{4n\trunc} (n\trunc + j) e^{-j}} \cdot \lprp{\frac{1}{\trunc^2} \cdot \frac{\Tr (\Sigma)}{n} + \frac{1}{\trunc}} \\
        &\leq \frac{C}{n} \cdot n\trunc \cdot \lprp{\frac{1}{\trunc^2} \cdot \frac{\Tr (\Sigma)}{n} + \frac{1}{\trunc}} \leq C \lprp{\frac{\Tr (\Sigma)}{\trunc n} + 1}
    \end{align*}
    which concludes the proof of the lemma.
\end{proof}

The next lemma is a refinement of a Gaussian rounding scheme by Depersin and Lecue. These improvements along with \cref{lem:spec_norm_bnd} allow us to establish that the stability-based estimators achieve sub-Gaussian recovery guarantees. On the other hand, in prior work, the guarantees incurred an additional multiplicative logarithmic factor. Before we proceed, we define the convex program analyzed in the proof:
\begin{equation*}
    \min_{w \in \mc{W}_\rho} \norm*{\sum_{i = 1}^n w_i z_iz_i^\top}. \tag{ROB-SDP} \label{eq:robsdp} 
\end{equation*}
For a dataset $\bm{Z}$ and $\rho$, let \ref{eq:robsdp}$(\bm{Z}, \rho)$ refer to the convex program above. The next lemma establishes a vectorization of the above program.

\begin{lemma}
    \label{lem:sdp_rounding}
    Let $\bm{Z} = \{z_1, \dots z_n\} \subset \R^d$ and $\rho \in [0, 1/2]$. Then:
    \begin{equation*}
        \min_{w \in \mc{W}_\rho} \norm*{\sum_{i = 1}^n w_i z_iz_i^\top} \leq 1024 \max_{\norm{v} = 1} \min_{w \in \mc{W}_{\rho / 4}} \sum_{i = 1}^n w_i \inp{v}{z_i}^2.
    \end{equation*}
\end{lemma}
\begin{proof}
    Note that \ref{eq:robsdp} may be recast as the following min-max program.
    \begin{equation*}
        \min_{w \in \mc{W}_\rho} \max_{\substack{M \succcurlyeq 0 \\ \Tr (M) = 1}} \inp*{M}{\sum_{i = 1}^n w_i z_iz_i^\top} = \max_{\substack{M \succcurlyeq 0 \\ \Tr (M) = 1}} \min_{w \in \mc{W}_\rho} \inp*{M}{\sum_{i = 1}^n w_i z_iz_i^\top} 
    \end{equation*}
    where the exchange of the min and max follows from von Neumann's minimax theorem. Let $M^*$ and $m^*$ denote the optimal solution and value to the program on the right. Consider now a Gaussian random vector $g$ drawn with mean $0$ and covariance $M$. Note that for any $i \in [n]$, $\inp{z_i}{g}$ is a Gaussian with mean $0$ and variance $z_i^\top M z_i$. Therefore, we have:
    \begin{equation*}
        \forall i \in [n]: \P \lbrb{\abs{\inp{z_i}{g}} \geq \frac{\sqrt{z_i^\top M z_i}}{4}} \geq \frac{3}{4}
    \end{equation*}
    from the fact that the pdf of a standard Gaussian random variable is bounded above by $1 / \sqrt{2\pi}$. Furthermore, we have from the fact that $\norm{M} \leq 1$ that:
    \begin{equation*}
        \P \lbrb{\norm{g} \leq 4} \geq \frac{9}{10}.
    \end{equation*}
    Hence, we get by the union bound that:
    \begin{equation}
        \label{eq:rounding_prob_bnd}
        \P \lbrb{\norm{g} \leq 4 \text{ and } \abs{\inp{z_i}{g}} \geq \frac{\sqrt{z_i^\top M z_i}}{4}} \geq \frac{1}{2}.
    \end{equation}
    For the rest of the proof, we divide into two cases based on the set:
    \begin{equation*}
        \mc{H} \coloneqq \lbrb{i: z_i^\top M z_i \geq \frac{m^*}{4 \rho}}.
    \end{equation*}
    The two cases we consider are as follows:
    \paragraph*{Case 1:} $\abs{\mc{H}} \geq \rho n$. In this case, consider $g$ which satisfies:
    \begin{equation*}
        \sum_{i \in \mc{H}} \bm{1} \lbrb{\abs{\inp{z_i}{g}} \geq \frac{1}{8} \sqrt{\frac{m^*}{\rho}} \text{ and } \norm{g} \leq 4} \geq \frac{\rho n}{2}.
    \end{equation*}
    Such a $g$ exists from \cref{eq:rounding_prob_bnd}. On this event, we have for the vector $\wt{g} = g / \norm{g}$:
    \begin{equation*}
        \min_{w \in \mc{W}_{\rho / 4}} \sum_{i = 1}^n w_i \inp{z_i}{\wt{g}}^2 \geq \sum_{i \in \mc{H}} w_i \inp{z_i}{\wt{g}}^2 \geq \frac{\rho}{4} \cdot \frac{1}{16} \cdot \frac{m^*}{64\rho} \geq \frac{m^*}{1024}.
    \end{equation*}
    This concludes the proof of the lemma in this case.
    
    \paragraph*{Case 2:} $\abs{\mc{H}} < \rho n$. Then, we must have by the minimax formulation:
    \begin{equation*}
        \frac{1}{n - \abs{\mc{H}}} \sum_{i \notin \mc{H}} z_i^\top M z_i \geq m^*
    \end{equation*}
    as this is a feasible solution to the minimax optimization problem. Now, consider the event:
    \begin{equation*}
        \frac{1}{n - \abs{\mc{H}}} \sum_{i \notin \mc{H}} z_i^\top M z_i \cdot \bm{1} \lbrb{\abs{\inp{z_i}{g}} \geq \frac{\sqrt{z_i^\top M z_i}}{4} \text{ and } \norm{g} \leq 4} \geq \frac{m^*}{2}.
    \end{equation*}
    Note that such a $g$ exists by the probabilistic method and \cref{eq:rounding_prob_bnd}. Furthermore, we have for any $\mc{G} \subseteq [n] \setminus \mc{H}$ with $\abs{\mc{G}} \leq \frac{\rho n}{4}$ from the definition of $\mc{H}$:
    \begin{equation*}
        \frac{1}{n - \abs{\mc{H}}} \sum_{i \in \mc{G}} z_i^\top M z_i \leq \frac{2}{n} \cdot \frac{\rho n}{4} \cdot \frac{m^*}{4 \rho} = \frac{m^*}{8}.
    \end{equation*}
    Hence, we get for any subset $\mc{I} \subset [n] \setminus \mc{H}$ with $\mc{I} \geq n - \abs{H} - \rho n / 4$:
    \begin{equation*}
        \frac{1}{n} \sum_{i \in \mc{I}} z_i^\top M z_i \bm{1} \lbrb{\abs{\inp{z_i}{g}} \geq \frac{\sqrt{z_i^\top M z_i}}{4}} \geq \frac{m^*}{4}.
    \end{equation*}
    Noticing that for $\wt{g} = g / \norm{g}$:
    \begin{equation*}
        z_i^\top M z_i \bm{1} \lbrb{\abs{\inp{z_i}{g}} \geq \frac{\sqrt{z_i^\top M z_i}}{4}} \leq 256 \inp{z_i}{\wt{g}}^2 \bm{1} \lbrb{\abs{\inp{z_i}{g}} \geq \frac{\sqrt{z_i^\top M z_i}}{4}},
    \end{equation*}
    we get that for all such $\mc{I}$:
    \begin{equation*}
        \frac{1}{n} \sum_{i \in \mc{I}} \inp{z_i}{\wt{g}}^2 \bm{1} \lbrb{\abs{\inp{z_i}{g}} \geq \frac{\sqrt{z_i^\top M z_i}}{4}} \geq \frac{m^*}{1024}.
    \end{equation*}
    We get as a consequence
    \begin{equation*}
        \min_{w \in \mc{W}_{\rho / 4}} \sum_{i = 1}^n w_i \inp{z_i}{\wt{g}}^2 \geq \min_{\substack{\mc{I} \subset [n] \setminus \mc{H} \\ \abs{\mc{I}} \geq n - \abs{\mc{H}} - \rho n / 4}} \frac{1}{n} \sum_{i \in \mc{I}} \inp{z_i}{\wt{g}}^2 \geq \frac{m^*}{1024}
    \end{equation*}
    which concludes this case and the proof of the lemma.
\end{proof}

The final technical result required to prove \cref{thm:sub_g_main} is the following which shows that the mean of the truncated data points when projected onto any direction is close to the true mean of the distribution. The proof is identical to that in \cite{lmtrimmed} and its proof is included in \cref{sec:deferred_proofs} for completion.
\begin{restatable}{lemma}{truncmeanbnd}
    \label{lem:trunc_mean_bnd}
    We have:
    \begin{equation*}
        \forall \norm{v} = 1: \sum_{i = 1}^n \phi_\tau (\inp{X_i}{v}) \leq C \cdot \lprp{\sqrt{n \Tr (\Sigma)} + n\sqrt{\trunc}}
    \end{equation*}
    with probability at least $1 - \exp (- \trunc n)$.
\end{restatable}

We are now ready to prove \cref{thm:sub_g_main}. We will condition on the conclusions of \cref{lem:trunc_mean_bnd,lem:quant_conc,lem:spec_norm_bnd,lem:wt_sp_bnd}. Note that this happens with probability at least $1 - \delta$. Observe that we have from \cref{lem:quant_conc}:
\begin{equation*}
    \forall \norm{v} = 1: \sum_{i = 1}^n \bm{1} \lbrb{\abs{\inp{X_i}{v}} \geq \tau} \leq \frac{\trunc n}{4}.
\end{equation*}
As a consequence, we get from \cref{lem:trunc_mean_bnd} for any $\norm{v} = 1$:
\begin{align*}
    \abs*{\sum_{i = 1}^n \inp{X_i}{v} \bm{1} \lbrb{\abs{\inp{X_i}{v}} \leq \tau}} &\leq C \lprp{\sqrt{n \Tr (\Sigma)} + n\sqrt{\trunc}} + \tau \sum_{i = 1}^n \bm{1} \lbrb{\abs{\inp{X_i}{v}} \geq \tau} \\
    &\leq C \cdot \lprp{\sqrt{n \Tr (\Sigma)} + n\sqrt{\trunc}}.
\end{align*}
Furthermore, we have from \cref{lem:spec_norm_bnd} that:
\begin{equation*}
    \sum_{i = 1}^n \inp{X_i}{v}^2 \bm{1} \lbrb{\abs{\inp{X_i}{v}} \leq \tau} \leq C \lprp{\frac{\Tr (\Sigma)}{\trunc} + n}.
\end{equation*}
As a consequence of the above, \cref{lem:sdp_rounding} yields for $x = \bm{0}$:
\begin{equation*}
    \min_{w \in \mc{W}_\trunc} \norm*{\sum_{i = 1}^n w_i (X_i - x)(X_i - x)^\top} \leq C \lprp{\frac{\Tr (\Sigma)}{\trunc n} + 1}.
\end{equation*}
Therefore, $\wt{\mu}$ satisfies:
\begin{equation*}
    \min_{w \in \mc{W}_\trunc} \norm*{\sum_{i = 1}^n w_i (X_i - \wt{\mu})(X_i - \wt{\mu})^\top} \leq C \lprp{\frac{\Tr (\Sigma)}{\trunc n} + 1}.
\end{equation*}
Let $w^*$ be the optimal solution to the above program and notice that $\wt{\mu} = \sum_{i = 1}^n w^*_i X_i$. Furthermore, let $w'$ be the uniform distribution over the set $\{i: \abs{\inp{X_i}{v}} \leq \tau\}$, $\mu'$ denote its mean and observe $w' \in \mc{W}_\tau$. Then, we have by \cref{lem:tv_var_mn_bnd} that:
\begin{equation*}
    \abs{\inp{\wt{\mu}}{v} - \mu'} \leq C \lprp{\sqrt{\frac{\Tr (\Sigma)}{n}} + \sqrt{\trunc}}.
\end{equation*}
As a consequence, since the above holds for all $\norm{v} = 1$, we have:
\begin{equation*}
    \norm{\wt{\mu} - \mu} = \norm{\wt{\mu}} \leq C \cdot \lprp{\sqrt{\frac{\Tr (\Sigma)}{n}} + \sqrt{\trunc}}.
\end{equation*}
Now, consider any $u$ such that $\norm{u} = 1$. We have
\begin{align*}
    \spread \lprp{\{\inp{X_i}{u}\}_{i = 1}^n, \inp{\wt{\mu}}{u}} &= \sqrt{\sum_{i = 1}^n w_i (\inp{X_i}{u} - \inp{\wt{\mu}}{u})^2} \leq \sqrt{2 \sum_{i = 1}^n w_i (\inp{X_i}{u}^2 + \inp{\wt{\mu}}{u}^2)} \\
    &\leq \sqrt{2} \cdot (\spread \lprp{\{\inp{X_i}{u}\}_{i = 1}^n, \inp{\wt{\mu}}{u}} + \norm{\wt{\mu}}) \leq C \cdot \lprp{\sqrt{\frac{\Tr (\Sigma)}{\trunc n}} + 1}.
\end{align*}
Therefore, we get:
\begin{equation*}
    \norm{\wh{\mu}} \leq \norm{\wt{\mu}} + \sqrt{\trunc} \cdot \spread \lprp{\{\inp{X_i}{u}\}_{i = 1}^n, \inp{\wt{\mu}}{u}} \leq C \lprp{\sqrt{\frac{\Tr (\Sigma)}{n}} + \sqrt{\frac{\log (1 / \delta)}{n}}}
\end{equation*}
thus establishing the theorem. \qed

\subsection{When is \cref{alg:sub_g_main} adversarially robust?}
\label{ssec:when_sub_g_adv_rob}

In this section, we establish the following result which proves that for \cref{alg:sub_g_main} to be adversarially robust, the set of points provided as input is required to have a stable set that includes an arbitrarily large fraction of the input points. Informally, an overwhelming fraction of outliers must either be filtered out, or their inclusion in the set of good points used to compute the mean does not substantially change the value of the estimate.

\begin{theorem}
    \label{thm:gen_conc}
    There exists an absolute constant $c' > 0$ and for any constant $c \in [0, 1]$, there exists an absolute constant $C > 0$ such that the following holds. Let $\bm{Z} = \{z_i\}_{i = 1}^n \subset \R^d$ be a pointset, $\trunc \in [0, c']$ and $\eta \geq 0$. Suppose there exists $z \in \R^d$ such that the outputs of \cref{alg:sub_g_main}, $\lbrb{\wh{\mu}_{+}, \wh{\mu}_{-}}$, when run with inputs $\bm{Z}$, $\delta = \exp (-\trunc n)$, and $s = \{+1, -1\}$ respectively satisfy:
    \begin{equation*}
        \max \lbrb{\norm{\wh{\mu}_{+} - z}, \norm{\wh{\mu}_{-} - z}} \leq \eta.
    \end{equation*}
    Then, there exists a subset $\mc{I} \subset [n]$ and $\abs{\mc{I}} \geq 1 - c \eta n$ satisfying:
    \begin{gather*}
        \frac{1}{\abs{\mc{I}}} \sum_{i \in \mc{I}} (z_i - \mu) (z_i - \mu)^\top \preccurlyeq C \frac{\eta^2}{\trunc} \quad \text{and} \quad \norm{\mu - z} \leq C \eta
        \quad \text{where} \quad \mu \coloneqq \frac{1}{\abs{\mc{I}}} \sum_{i = 1}^n z_i.
    \end{gather*}
\end{theorem}
\begin{proof}
    Noting that $\wh{\mu}_{+} = \wt{\mu} + \sqrt{\eps} \sigma_v v$ and $\wh{\mu}_{-} = \wt{\mu} - \sqrt{\eps} \sigma_v v$, we must have $\norm{\wt{\mu} - z} \leq \eta$ by convexity. Now, we have by the parallelogram law:
    \begin{equation*}
        2 \eta^2 \geq \norm{\wh{\mu}_{+} - z}^2 + \norm{\wh{\mu}_{-} - z}^2 = 2(\norm{\wt{\mu} - z}^2 + \eps \sigma_v^2) \implies \sigma_v^2 \leq \frac{\eta^2}{\eps}. 
    \end{equation*}
    Notice that in the computation of $\spread$, for any $c_1 \trunc n / 16 \leq i \leq n - c_1 \trunc n / 16$, we have $c \leq \wt{w}_{i} \leq 1$ and as a consequence, we have:
    \begin{equation*}
        cn \leq \sum_{i = 1}^n \wt{w}_i \leq n \implies \forall\, \frac{c_1 \trunc n}{16} \leq i \leq n - \frac{c_1 \trunc n}{16}: \frac{c}{n} \leq w_i \leq \frac{1}{cn}.
    \end{equation*}
    As a consequence, we have for $\norm{v} = 1$:
    \begin{equation*}
        \min_{w \in \mc{W}_{c_1 \trunc / 16}} \sum_{i = 1}^n w_i \inp{v}{z_i - \wt{\mu}}^2 \leq C \cdot \frac{\eta^2}{\trunc}.
    \end{equation*}
    From \cref{lem:sdp_rounding}, we also get that:
    \begin{equation*}
        \min_{w \in \mc{W}_{c_1 \trunc / 4}} \sum_{i = 1}^n w_i (z_i - \wt{\mu}) (z_i - \wt{\mu})^\top \preccurlyeq C \frac{\eta^2}{\trunc}.
    \end{equation*}
    Let $w^*$ be the solution to the above problem. Now, consider the set:
    \begin{equation*}
        \mc{I} \coloneqq \lbrb{i: w^*_i \geq \frac{1}{2(1 - c_1 \trunc / 4) n}}.
    \end{equation*}
    To bound the size of $\mc{I}$, observe that:
    \begin{equation*}
        \abs{\mc{I}} \cdot \frac{1}{(1 - c_1 \trunc / 4)n} + (n - \abs{\mc{I}}) \cdot \frac{1}{2 (1 - c_1 \trunc / 4)n} \geq 1 \implies \abs{\mc{I}} \geq \lprp{1 - \frac{c_1 \trunc}{2}} n. 
    \end{equation*}
    Hence, we get for this set by the definition of $\mc{I}$:
    \begin{equation*}
        \frac{1}{\abs{\mc{I}}} \sum_{i \in \mc{I}} w_i (z_i - \mu) (z_i - \mu)^\top \preccurlyeq \frac{1}{\abs{\mc{I}}} \sum_{i \in \mc{I}} w_i (z_i - \wt{\mu}) (z_i - \wt{\mu})^\top \preccurlyeq 2 \sum_{i = 1}^n w_i (z_i - \wt{\mu})(z_i - \wt{\mu})^\top \preccurlyeq C\frac{\eta^2}{\trunc}.
    \end{equation*}
    Finally, to conclude the proof, note that there exists $w^\dagger \in \mc{W}_\trunc$ such that:
    \begin{equation*}
        w^\dagger \coloneqq \argmin_{w \in \mc{W}_\trunc} \norm*{\sum_{i = 1}^n w_i (z_i - \wt{\mu}) (z_i - \wt{\mu})^\top} \text{ and } \wt{\mu} = \sum_{i = 1}^n w^\dagger_i z_i.
    \end{equation*}
    Note, also that:
    \begin{equation*}
        \sum_{i = 1}^n \norm*{w^\dagger_i (z_i - \wt{\mu}) (z_i - \wt{\mu})^\top} \preccurlyeq C \frac{\eta^2}{\trunc}.
    \end{equation*}
    Observe, now, that the uniform distribution over $\mc{I}$, $w_{\mc{I}}$ and the distribution $w^\dagger$ both satisfy $w_{\mc{I}}, w^\dagger \in \mc{W}_{\trunc}$. From \cref{lem:tv_wr_bnd} and applying \cref{lem:tv_var_mn_bnd} we get:
    \begin{equation*}
        \forall \norm{v} = 1: \abs{\inp{v}{\mu} - \inp{v}{\wt{\mu}}} \leq C \eta.
    \end{equation*}
    As a consequence, we obtain:
    \begin{equation*}
        \norm{\mu - \wt{\mu}} \leq C \eta
    \end{equation*}
    concluding the proof of the theorem.
\end{proof}

\paragraph{Acknowledgements:} The authors would like to thank Sam Hopkins for proposing the project as part of a course on Algorithmic Statistics.

\bibliographystyle{alpha}
\bibliography{refs}

\newcommand{\etalchar}[1]{$^{#1}$}
\begin{thebibliography}{DKK{\etalchar{+}}17}

\bibitem[AMS99]{alon1999space}
Noga Alon, Yossi Matias, and Mario Szegedy.
\newblock The space complexity of approximating the frequency moments.
\newblock volume~58, pages 137--147. 1999.
\newblock Twenty-eighth Annual ACM Symposium on the Theory of Computing (Philadelphia, PA, 1996).

\bibitem[Ash90]{ash1990information}
Robert~B. Ash.
\newblock {\em Information Theory}.
\newblock Dover Publications, 1990.

\bibitem[AZ24]{azcov}
Pedro Abdalla and Nikita Zhivotovskiy.
\newblock Covariance estimation: Optimal dimension-free guarantees for adversarial corruption and heavy tails.
\newblock {\em Journal of the European Mathematical Society}, 2024.

\bibitem[BLM13]{boucheron2013concentration}
St\'{e}phane Boucheron, G\'{a}bor Lugosi, and Pascal Massart.
\newblock {\em Concentration inequalities}.
\newblock Oxford University Press, Oxford, 2013.
\newblock A nonasymptotic theory of independence, With a foreword by Michel Ledoux.

\bibitem[Bou02]{bousquetthesis}
Olivier Bousquet.
\newblock {\em Concentration Inequalities and Empirical Processes Theory Applied to the Analysis of Learning Algorithms}.
\newblock PhD thesis, Ecole Polytechnique, 2002.

\bibitem[CG17]{catoni2018dimension}
O.~Catoni and I.~Giulini.
\newblock Dimension-free {PAC}-{B}ayesian bounds for the estimation of the mean of a random vector.
\newblock {\em NIPS 2017 Workshop; (Almost) 50 Shades of Bayesian Learning: PAC-Bayesian Trends and Insights}, 2017.

\bibitem[CTBJ22]{ctbj}
Yeshwanth Cherapanamjeri, Nilesh Tripuraneni, Peter~L. Bartlett, and Michael~I. Jordan.
\newblock Optimal mean estimation without a variance.
\newblock In Po{-}Ling Loh and Maxim Raginsky, editors, {\em Conference on Learning Theory, 2-5 July 2022, London, {UK}}, volume 178 of {\em Proceedings of Machine Learning Research}, pages 356--357. {PMLR}, 2022.

\bibitem[Din16]{DBLP:conf/focs/2016}
Irit Dinur, editor.
\newblock {\em {IEEE} 57th Annual Symposium on Foundations of Computer Science, {FOCS} 2016, 9-11 October 2016, Hyatt Regency, New Brunswick, New Jersey, {USA}}. {IEEE} Computer Society, 2016.

\bibitem[DK19]{diakonikolas2019recent}
Ilias Diakonikolas and Daniel~M Kane.
\newblock Recent advances in algorithmic high-dimensional robust statistics.
\newblock {\em arXiv preprint arXiv:1911.05911}, 2019.

\bibitem[DKK{\etalchar{+}}16]{diakonikolas2016robust}
Ilias Diakonikolas, Gautam Kamath, Daniel~M. Kane, Jerry Li, Ankur Moitra, and Alistair Stewart.
\newblock Robust estimators in high dimensions without the computational intractability.
\newblock In Dinur \cite{DBLP:conf/focs/2016}, pages 655--664.

\bibitem[DKK{\etalchar{+}}17]{diakonikolas2017being}
Ilias Diakonikolas, Gautam Kamath, Daniel~M. Kane, Jerry Li, Ankur Moitra, and Alistair Stewart.
\newblock Being robust (in high dimensions) can be practical.
\newblock In Doina Precup and Yee~Whye Teh, editors, {\em Proceedings of the 34th International Conference on Machine Learning, {ICML} 2017, Sydney, NSW, Australia, 6-11 August 2017}, volume~70 of {\em Proceedings of Machine Learning Research}, pages 999--1008. {PMLR}, 2017.

\bibitem[DKP20]{dkp20}
Ilias Diakonikolas, Daniel~M. Kane, and Ankit Pensia.
\newblock Outlier robust mean estimation with subgaussian rates via stability.
\newblock In Hugo Larochelle, Marc'Aurelio Ranzato, Raia Hadsell, Maria{-}Florina Balcan, and Hsuan{-}Tien Lin, editors, {\em Advances in Neural Information Processing Systems 33: Annual Conference on Neural Information Processing Systems 2020, NeurIPS 2020, December 6-12, 2020, virtual}, 2020.

\bibitem[DKP21]{diakonikolasOutlierRobustMean2021}
Ilias Diakonikolas, Daniel~M. Kane, and Ankit Pensia.
\newblock Outlier {{Robust Mean Estimation}} with {{Subgaussian Rates}} via {{Stability}}.
\newblock (arXiv:2007.15618), March 2021.

\bibitem[DL19]{deplecmean}
Jules Depersin and Guillaume Lecué.
\newblock Robust subgaussian estimation of a mean vector in nearly linear time, 2019.

\bibitem[HLZ20a]{hopkinsRobustHeavyTailedMean2020}
Samuel~B. Hopkins, Jerry Li, and Fred Zhang.
\newblock Robust and {{Heavy-Tailed Mean Estimation Made Simple}}, via {{Regret Minimization}}.
\newblock {\em arXiv.org}, July 2020.

\bibitem[HLZ20b]{hlz20}
Samuel~B. Hopkins, Jerry Li, and Fred Zhang.
\newblock Robust and heavy-tailed mean estimation made simple, via regret minimization.
\newblock In Hugo Larochelle, Marc'Aurelio Ranzato, Raia Hadsell, Maria{-}Florina Balcan, and Hsuan{-}Tien Lin, editors, {\em Advances in Neural Information Processing Systems 33: Annual Conference on Neural Information Processing Systems 2020, NeurIPS 2020, December 6-12, 2020, virtual}, 2020.

\bibitem[Hop20]{hopkins2018sub}
Samuel~B. Hopkins.
\newblock Mean estimation with sub-{G}aussian rates in polynomial time.
\newblock {\em Ann. Statist.}, 48(2):1193--1213, 2020.

\bibitem[Hub64]{huber64}
Peter~J. Huber.
\newblock Robust estimation of a location parameter.
\newblock {\em Ann. Math. Statist.}, 35:73--101, 1964.

\bibitem[JVV86]{jerrum1986random}
Mark~R. Jerrum, Leslie~G. Valiant, and Vijay~V. Vazirani.
\newblock Random generation of combinatorial structures from a uniform distribution.
\newblock {\em Theoret. Comput. Sci.}, 43(2-3):169--188, 1986.

\bibitem[LM19]{lugosi2017sub}
G\'{a}bor Lugosi and Shahar Mendelson.
\newblock Sub-{G}aussian estimators of the mean of a random vector.
\newblock {\em Ann. Statist.}, 47(2):783--794, 2019.

\bibitem[LM20]{lugosi2018risk}
G\'{a}bor Lugosi and Shahar Mendelson.
\newblock Risk minimization by median-of-means tournaments.
\newblock {\em J. Eur. Math. Soc. (JEMS)}, 22(3):925--965, 2020.

\bibitem[LM21]{lmtrimmed}
G\'{a}bor Lugosi and Shahar Mendelson.
\newblock Robust multivariate mean estimation: the optimality of trimmed mean.
\newblock {\em Ann. Statist.}, 49(1):393--410, 2021.

\bibitem[LRV16]{lai2016agnostic}
Kevin~A. Lai, Anup~B. Rao, and Santosh~S. Vempala.
\newblock Agnostic estimation of mean and covariance.
\newblock In Dinur \cite{DBLP:conf/focs/2016}, pages 665--674.

\bibitem[LT11]{ledoux1991probability}
Michel Ledoux and Michel Talagrand.
\newblock {\em Probability in {B}anach spaces}.
\newblock Classics in Mathematics. Springer-Verlag, Berlin, 2011.
\newblock Isoperimetry and processes, Reprint of the 1991 edition.

\bibitem[Mas90]{massartTightConstantDvoretzkyKieferWolfowitz1990}
P.~Massart.
\newblock The {{Tight Constant}} in the {{Dvoretzky-Kiefer-Wolfowitz Inequality}}.
\newblock {\em The Annals of Probability}, 18(3):1269--1283, July 1990.

\bibitem[McD89]{mcdiarmid}
Colin McDiarmid.
\newblock On the method of bounded differences.
\newblock In {\em Surveys in combinatorics, 1989 ({N}orwich, 1989)}, volume 141 of {\em London Math. Soc. Lecture Note Ser.}, pages 148--188. Cambridge Univ. Press, Cambridge, 1989.

\bibitem[MVZ21]{mvz}
Jaouad Mourtada, Tomas Va\v{s}kevi\v{c}ius, and Nikita Zhivotovskiy.
\newblock Distribution-free robust linear regression.
\newblock {\em Math. Stat. Learn.}, 4(3-4):253--292, 2021.

\bibitem[MZ20]{mzcovariance}
Shahar Mendelson and Nikita Zhivotovskiy.
\newblock Robust covariance estimation under {$L_4-L_2$} norm equivalence.
\newblock {\em Ann. Statist.}, 48(3):1648--1664, 2020.

\bibitem[NY83]{NemYud83}
Arkadi~S. Nemirovsky and David~B. Yudin.
\newblock {\em Problem complexity and method efficiency in optimization}.
\newblock Wiley-Interscience Series in Discrete Mathematics. John Wiley \& Sons, Inc., New York, 1983.
\newblock Translated from the Russian and with a preface by E. R. Dawson.

\bibitem[SCV18]{steinhardt2017resilience}
Jacob Steinhardt, Moses Charikar, and Gregory Valiant.
\newblock Resilience: {A} criterion for learning in the presence of arbitrary outliers.
\newblock In Anna~R. Karlin, editor, {\em 9th Innovations in Theoretical Computer Science Conference, {ITCS} 2018, January 11-14, 2018, Cambridge, MA, {USA}}, volume~94 of {\em LIPIcs}, pages 45:1--45:21. Schloss Dagstuhl - Leibniz-Zentrum f{\"{u}}r Informatik, 2018.

\bibitem[Tal94]{talagrandsharper}
Michel Talagrand.
\newblock Sharper bounds for {G}aussian and empirical processes.
\newblock {\em Ann. Probab.}, 22(1):28--76, 1994.

\bibitem[Tal96]{talagrandnewconcentration}
Michel Talagrand.
\newblock New concentration inequalities in product spaces.
\newblock {\em Invent. Math.}, 126(3):505--563, 1996.

\end{thebibliography}

\appendix

\section{Auxiliary Technical Results}

\subsection{Probability and Empirical Process Theory}

We present Bousquet's inequality on the suprema of empirical processes \cite{bousquetthesis} which builds on prior results by Talagrand \cite{talagrandsharper,talagrandnewconcentration}.

\begin{theorem}[\cite{bousquetthesis,boucheron2013concentration}]
    \label{thm:bousequet_thm}
    Let $X_1, \dots, X_n$ be independent identically distributed random vectors indexed by an index set $\mc{T}$. Assume that $\E [X_{i, s}] = 0$, and $X_{i, s} \leq 1$ for all $s \in \mc{T}$. Let $Z = \sup_{s \in \mc{T}} \sum_{i = 1}^n X_{i, s}$, $\nu = 2 \E Z + \sigma^2$ where $\sigma^2 = \sup_{s \in \mc{T}} \sum_{i = 1}^n \E X_{i, s}^2$ is the wimpy variance. Let $\phi(u) = e^u - u - 1$ and $h(u) = (1 + u)\log (1 + u) - u$, for $u \geq -1$. Then for all $\lambda \geq 0$,
    \begin{equation*}
        \log \E e^{\lambda (Z - \E Z)} \leq \nu \phi(\lambda).
    \end{equation*}
    Also, for all $t \geq 0$,
    \begin{equation*}
        \P \lbrb{Z \geq \E Z + t} \leq e^{-\nu h(t / \nu)} \leq \exp \lprp{- \frac{t^2}{2 (\nu + t / 3)}}.
    \end{equation*}
\end{theorem}

We also recall McDiarmid's classical bounded differences inequality.

\begin{theorem}[\cite{mcdiarmid,boucheron2013concentration}]
    \label{thm:mcdiarmid}
    Let $n \in \N$, $\mc{X}$ denote some domain and assume $f: \mc{X}^n \to \R$ satisfies:
    \begin{equation*}
        \forall i \in [n]: \sup_{\substack{x_1, \dots, x_n \\ x^\prime_i \in \mc{X}}} \abs{f (x_1, \dots, x_n) - f (x_1, \dots, x_i^\prime, \dots, x_n)} \leq 1.
    \end{equation*}
    Then
    \begin{equation*}
        \Pr \lbrb{f(X) - \E f(X) \geq t} \leq e^{-2 t^2 / n}.
    \end{equation*}
\end{theorem}

We also require the Ledoux-Talagrand contraction inequality \cite{ledoux1991probability} (as stated in \cite{boucheron2013concentration}). 

\begin{theorem}[\cite{ledoux1991probability,boucheron2013concentration}]
    \label{thm:ledtal}
    Let $x_1, \dots, x_n$ be vectors whose real-valued components are indexed by $\mc{T}$, that is, $x_i = (x_{i, s})_{s \in \mc{T}}$. For each $i = 1, \dots, n$, let $\phi_i: \R \to \R$ be a $1$-Lipschitz function such that $\phi_i (0) = 0$. Let $\eps_1, \dots, \eps_n$ be independent Rademacher random variables, and let $\Psi: [0, \infty) \to \R$ be a non-decreasing convex function. Then,
    \begin{equation*}
        \E \lsrs{\Psi \lprp{\sup_{s \in \mc{T}} \sum_{i = 1}^n \eps_i \phi_i (x_{i, s})}} \leq \E \lsrs{\Psi \lprp{\sup_{s \in \mc{T}} \sum_{i = 1}^n \eps_i x_{i, s}}}
    \end{equation*}
    and 
    \begin{equation*}
        \E \lsrs{\Psi \lprp{\frac{1}{2} \sup_{s \in \mc{T}} \abs*{\sum_{i = 1}^n \eps_i \phi_i (x_{i, s})}}} \leq \E \lsrs{\Psi \lprp{\sup_{s \in \mc{T}} \abs*{\sum_{i = 1}^n \eps_i x_{i, s}}}}.
    \end{equation*}
\end{theorem}

We will use the following simple corollary of the second conclusion in our proofs.

\begin{corollary}
    \label{cor:ledtal}
    Assume the setting of \cref{thm:ledtal}. Then,
    \begin{equation*}
        \E \lsrs{\sup_{s \in \mc{T}} \abs*{\sum_{i = 1}^n \eps_i \phi_i (x_{i, s})}} \leq 2 \E \lsrs{\sup_{s \in \mc{T}} \abs*{\sum_{i = 1}^n \eps_i x_{i, s}}}.
    \end{equation*}
\end{corollary}

\begin{theorem}[\cite{massartTightConstantDvoretzkyKieferWolfowitz1990}]\label{thm:DKW}
    Let $X_1,\dots,X_n$ be real-valued, iid random variables with CDF $F$. Let $F_n$ denote the empirical distribution function: $$F_n(x)=\frac{1}{n}\sum_{i=1}^n\mathbb{1}[X_i\leq x]$$
    For any $\epsilon > 0$: $$\P\lbrb{\sup_{x \in \R}|F_n(x)-F(x)|>\epsilon}\leq 2e^{-2n\epsilon^2}$$
\end{theorem}

\begin{theorem}[\cite{ash1990information} Lemma 4.7.2]\label{thm:anticoncentration}
    Let $S\sim B(n,p)$ be the sum of $n$ independent $p$-biased coins. Then: $$\P\lbrb{S\geq \lambda n}\geq\frac{1}{\sqrt{8\lambda n(1-\lambda)}}e^{-nD\lprp{\lambda||p}}$$
    where $D(\lambda||p)=\lambda\log\frac{\lambda}{p} + (1-\lambda)\log\frac{1-\lambda}{1-p}$
\end{theorem}
\begin{corollary}\label{cor:anticoncentration}
    Let $S\sim B(n,p)$ be the sum of $n$ independent $p$-biased coins. Then: $$\P\lbrb{\frac{S}{n}\geq p + \epsilon}\geq\frac{1}{\sqrt{2n}}e^{-\frac{n\epsilon^2}{p(1-p)}}$$
\end{corollary}
\begin{proof}
    We first observe that $\lambda(1-\lambda)\geq \frac{1}{4}$. Therefore, $$\frac{1}{\sqrt{8\lambda n(1-\lambda)}}\geq\frac{1}{\sqrt{2n}}$$
    Next, we show that $D(\lambda||p)\leq \epsilon^2$ where $\lambda=p+\epsilon$
    \begin{align*}
        \lambda\log\frac{\lambda}{p} + (1-\lambda)\log\frac{1-\lambda}{1-p}
        &\leq \lambda\lprp{\frac{\lambda}{p}-1} + (1-\lambda)\lprp{\frac{1-\lambda}{1-p}-1}\\
        &= \frac{\lambda(\lambda-p)}{p}+\frac{(1-\lambda)(p-\lambda)}{1-p}\\
        &= \frac{\lambda(\lambda-p)(1-p) + (1-\lambda)(p-\lambda)p}{p(1-p)}\\
        &= \frac{\epsilon\lambda(1-p)-\epsilon p(1-\lambda)}{p(1-p)}\\
        &= \frac{\epsilon^2}{p(1-p)}\\
    \end{align*}
\end{proof}

\subsection{Tools from Robust Statistics}

\begin{lemma}
    \label{lem:tv_wr_bnd}
    For any $\rho \in [0, 1/2]$ and for any $w, w' \in \mc{W}_\rho$, we have:
    \begin{equation*}
        \tv (w, w') \leq 2 \rho.
    \end{equation*}
\end{lemma}
\begin{proof}
    Observe that:
    \begin{equation*}
        \tv(w, w') = \sum_{i = 1}^n \max (w_i, w'_i) - 1 \leq \frac{1}{1 - \rho} - 1 = \frac{\rho}{1 - \rho} \leq 2\rho
    \end{equation*}
    and the lemma follows.
\end{proof}

\begin{lemma}
    \label{lem:tv_var_mn_bnd}
    Let $\bm{Y} = \{y_1, \dots, y_n\} \subset \R$ and $\rho \in [0, 1/4]$. Now, let $w, w' \in \mc{W}_\rho$ be such that:
    \begin{equation*}
        \sum_{i = 1}^n w_i (y_i - \mu)^2 \leq \sigma^2,\ \sum_{i = 1}^n w_i (y_i - \mu')^2 \leq \sigma'^2 \text{ where } \mu = \sum_{i = 1}^n w_i y_i,\ \mu' = \sum_{i = 1}^n w'_i y_i.
    \end{equation*}
    Then, we have:
    \begin{equation*}
        \abs{\mu - \mu'} \leq 2 \sqrt{\rho} (\sigma + \sigma').
    \end{equation*}
\end{lemma}
\begin{proof}
    Note that we may assume $w \neq w'$ as this case is trivial. Let $\wt{w}$ be the distribution be such that $\wt{w}_i \propto \min (w_i, w'_i)$. Furthermore, let $\wh{w}$ and $\wh{w}'$ distributions such that $\wh{w}_i \propto w_i - \min (w_i, w'_i)$ and $\wh{w}'_i \propto w'_i - \min (w_i, w'_i)$. Note that $\wt{w}$ is well defined as $\tv (w, w') \leq 1/2$ from \cref{lem:tv_wr_bnd} and $\wh{w}, \wh{w}'$ are well-defined from the fact that $\tv (w, w') > 0$. Letting $\nu = \tv (w, w')$:
    \begin{equation*}
        w = (1 - \nu) \wt{w} + \nu \wh{w},\ w' = (1 - \nu) \wt{w} + \nu \wh{w}.
    \end{equation*}
    Letting $\wt{\mu}, \wh{\mu}, \wh{\mu}'$ be the means of $\wt{w}$, $\wh{w}$, and $\wh{w}'$ respectively, we have:
    \begin{gather*}
        \abs{\mu - \wt{\mu}} = \nu \abs{\wh{\mu} - \wt{\mu}} \\
        \abs{\mu' - \wt{\mu}} = \nu \abs{\wh{\mu}' - \wt{\mu}}.
    \end{gather*}
    As a consequence, by the law of total variation:
    \begin{gather*}
        \nu (\wh{\mu} - \mu)^2 + (1 - \nu) (\wt{\mu} - \mu)^2 = \nu (1 - \nu) (\wt{\mu} - \wh{\mu})^2 \leq \sigma^2 \\
        \nu (\wh{\mu}' - \mu')^2 + (1 - \nu) (\wt{\mu} - \mu')^2 = \nu (1 - \nu) (\wt{\mu} - \wh{\mu}')^2 \leq \sigma'^2.
    \end{gather*}
    Therefore, we get:
    \begin{gather*}
        \abs{\wt{\mu} - \wh{\mu}} \leq \frac{\sqrt{2}\sigma}{\sqrt{\nu}} \implies \abs{\mu - \wt{\mu}} \leq \sqrt{2 \nu} \sigma\\
        \abs{\wt{\mu} - \wh{\mu}'} \leq \frac{\sqrt{2}\sigma'}{\sqrt{\nu}} \implies \abs{\mu' - \wt{\mu}} \leq \sqrt{2 \nu} \sigma'.
    \end{gather*}
    By the triangle inequality, we get:
    \begin{equation*}
        \abs{\mu - \mu'} \leq 2 \sqrt{\rho} (\sigma + \sigma')
    \end{equation*}
    concluding the proof of the lemma.
\end{proof}

\section{Deferred Proofs from \cref{sec:sub_g}}
\label{sec:deferred_proofs}

The deferred proofs of technical results utilized in the proof of \cref{thm:sub_g_main} in \cref{sec:sub_g} are collected here. The first is the proof of \cref{lem:trunc_mean_bnd}.
\truncmeanbnd*
\begin{proof}
    We proceed as follows. Consider the random variable:
    \begin{equation*}
        Z \coloneqq \max_{\norm{v} = 1} W_{i, v} \text{ where } W_{i, v} \coloneqq \sum_{i = 1}^n \phi_\tau (\inp{X_i}{v}) - \E_{X \thicksim \mc{D}} [\phi_\tau (\inp{X}{v})].
    \end{equation*}
    We have where $X'_i$ and $\gamma_i$ are independent copies of $X_i$ and independent Rademacher random variables respectively:
    \begin{align*}
        \E [Z] &= \E \lsrs{\max_{\norm{v} = 1} \sum_{i = 1}^n \phi_\tau (\inp{X_i}{v}) - \E [\phi_\tau (\inp{X_i'}{v})]} \\
        &\leq \E_{X_i, X_i'} \lsrs{\max_{\norm{v} = 1} \sum_{i = 1}^n \phi_\tau (\inp{X_i}{v}) - \phi_\tau (\inp{X_i}{v})} \\
        &= \E_{X_i, X_i', \gamma_i} \lsrs{\max_{\norm{v} = 1} \sum_{i = 1}^n \gamma_i(\phi_\tau (\inp{X_i}{v}) - \phi_\tau (\inp{X_i}{v}))} \\
        &\leq 2 \E_{X_i, \gamma_i} \lsrs{\max_{\norm{v} = 1} \sum_{i = 1}^n \gamma_i\phi_\tau (\inp{X_i}{v})} \\
        &\leq 4 \E_{X_i, \gamma_i} \lsrs{\max_{\norm{v} = 1} \sum_{i = 1}^n \gamma_i \inp{X_i}{v}} \leq 4 \E_{X_i, \gamma_i} \lsrs{\norm*{\sum_{i = 1}^n \gamma_i X_i}} \\
        &\leq 4 \sqrt{\E_{X_i, \gamma_i} \lsrs{\norm*{\sum_{i = 1}^n \gamma_i X_i}^2}} = 4 \sqrt{n \Tr (\Sigma)},
    \end{align*}
    the third inequality following from Ledoux-Talagrand contraction (\cref{cor:ledtal}). Furthermore, note that for all $\norm{v} = 1$:
    \begin{equation*}
        \E [W_{i, v}^2] \leq \E [\phi_\tau (\inp{X_i}{v})^2] \leq \E [\inp{X_i}{v}^2] \leq 1
    \end{equation*}
    and that $W_{i, v} \leq 2\tau$ almost surely. Therefore, we get from \cref{thm:bousequet_thm} that:
    \begin{equation*}
        \P \lbrb{Z \geq r} \leq \exp \lprp{- \frac{r^2}{64} \cdot \frac{1}{\tau \sqrt{n \Tr (\Sigma)} + n + r\tau}}.
    \end{equation*}
    By picking $r = C n\trunc \tau$, we get that:
    \begin{equation*}
        \P \lbrb{Z \geq r} \leq \exp (- n\trunc).
    \end{equation*}
    By observing that for all $\norm{v} = 1$ and $X \thicksim \mc{D}$,
    \begin{align*}
        \abs{\E [\phi_\tau (\inp{X}{v})]} &= \abs{\E [\inp{X}{v}] - \E [(\inp{X}{v} - \sgn (\inp{X}{v}) \tau) \bm{1} \lbrb{\abs{\inp{X}{v}} \geq \tau}]} \\
        &= \abs{\E [(\inp{X}{v} - \sgn (\inp{X}{v}) \tau) \bm{1} \lbrb{\abs{\inp{X}{v}} \geq \tau}]} \\
        &\leq \abs{\E [\inp{X}{v} \bm{1} \lbrb{\abs{\inp{X}{v}} \geq \tau}]} + \tau \P \lbrb{\abs{\inp{X}{v}} \geq \tau} \\
        &\leq \sqrt{\E [\inp{X}{v}^2]} \sqrt{\P \lbrb{\abs{\inp{X}{v}} \geq \tau}} + \tau \cdot \frac{1}{\tau^2} \leq \frac{2}{\tau} \leq C \sqrt{\trunc}
    \end{align*}
    the lemma follows.
\end{proof}

Next, we present the proof of \cref{lem:spec_norm_bnd}.
\specnrmbnd*
\begin{proof}
    We may assume without loss of generality that $\norm{\Sigma} = 1$, First, observe that:
    \begin{equation*}
        \forall \norm{v} = 1 : \E_{X \thicksim \mc{D}} [\phi_\tau (\inp{X_i}{v})^2] \leq \E_{X \thicksim \mc{D}} [\inp{X_i}{v}^2] \leq 1.
    \end{equation*}
    And, as a consequence, we get where $X'_i$ and $\gamma_i$ are independent copies of $X_i$ and independent Rademacher random variables respectively:
    \begin{align*}
        \E \lsrs{\max_{\norm{v} = 1} \sum_{i = 1}^n \phi_\tau (\inp{X_i}{v})^2} &\leq \E \lsrs{\max_{\norm{v} = 1} \sum_{i = 1}^n \phi_\tau (\inp{X_i}{v})^2 - \E \lsrs{\phi_\tau (\inp{X_i'}{v})^2}} + \max_{\norm{v} = 1} n \E \lsrs{\phi_\tau (\inp{X}{v})^2} \\
        &\leq n + \E \lsrs{\max_{\norm{v} = 1} \sum_{i = 1}^n \phi_\tau (\inp{X_i}{v})^2 - \E \lsrs{\phi_\tau (\inp{X_i'}{v})^2}} \\
        &\leq n + \E_{X_i, X_i'} \lsrs{\max_{\norm{v} = 1} \sum_{i = 1}^n \phi_\tau (\inp{X_i}{v})^2 - \phi_\tau (\inp{X_i'}{v})^2} \\
        &\leq n + \E_{X_i, X_i', \gamma_i} \lsrs{\max_{\norm{v} = 1} \sum_{i = 1}^n \gamma_i (\phi_\tau (\inp{X_i}{v})^2 - \phi_\tau (\inp{X_i'}{v})^2)} \\
        &\leq n + 2 \E_{X_i, \gamma_i} \lsrs{\max_{\norm{v} = 1} \sum_{i = 1}^n \gamma_i \phi_\tau (\inp{X_i}{v})^2} \\
        &\leq n + 8 \tau \E_{X_i, \gamma_i} \lsrs{\max_{\norm{v} = 1} \sum_{i = 1}^n \gamma_i \inp{X_i}{v}} \leq n + 8 \tau \E_{X_i, \gamma_i} \lsrs{\norm{\sum_{i = 1}^n \gamma_i X_i}} \\
        &\leq n + 8 \tau \sqrt{\E_{X_i, \gamma_i} \lsrs{\norm*{\sum_{i = 1}^n \gamma_i X_i}^2}} = n + 8 \tau \sqrt{n \Tr (\Sigma)},
        \numberthis \label{eq:spec_bn}
    \end{align*}
    the sixth inequality following from Ledoux-Talagrand contraction (\cref{cor:ledtal}) and noting that $\phi_\tau$ is $2\tau$-Lipschitz. To establish high-probability concentration, define:
    \begin{equation*}
        Z = \max_{\norm{v} = 1} \sum_{i = 1}^n W_{i, v} \text{ where } W_{i, v} \coloneqq \phi_\tau (\inp{X_i}{v})^2 - \E_{X \thicksim \mc{D}} [\phi_\tau (\inp{X}{v})^2].
    \end{equation*}
    We get as a consequence of \cref{eq:spec_bn} that:
    \begin{equation*}
        \E [Z] \leq 8\tau \sqrt{n \Tr (\Sigma)}.
    \end{equation*}
    Furthermore, we get:
    \begin{equation*}
        \forall \norm{v} = 1: \E [W_{i, v}^2] \leq \E [\phi_\tau (\inp{X_i}{v})^4] \leq \tau^2 \E [\phi_\tau (\inp{X_i}{v})^2] \leq \tau^2.
    \end{equation*}
    Finally, note that almost surely:
    \begin{equation*}
        W_{i, v} \leq \tau^2.
    \end{equation*}
    We now obtain via \cref{thm:bousequet_thm} by renormalizing $Y_{i, v} \coloneqq W_{i, v} / \tau^2$:
    \begin{equation*}
        \P \lbrb{Z \geq r} \leq \exp \lprp{- \frac{r^2}{32 (\tau^3 \sqrt{n \Tr \Sigma} + n\tau^2 + r\tau^2}}.
    \end{equation*}
    Now, by the following setting for $r$:
    \begin{equation*}
        r \coloneqq C \max \lbrb{\frac{\Tr \Sigma}{\trunc}, n},
    \end{equation*}
    and the above inequality, we get that:
    \begin{equation*}
        \P \lbrb{Z \geq r} \leq \exp (- \trunc n).
    \end{equation*}
    By observing that:
    \begin{equation*}
        \max_{v} \sum_{i = 1}^n \phi_\tau (\inp{X_i}{v})^2 \leq Z + n,
    \end{equation*}
    the lemma follows.
\end{proof}

\end{document}